\documentclass{tglat2e}
\usepackage{zref-xr,hyperref,pb-diagram,amssymb,epic,eepic,verbatim,graphicx,graphics,epsfig,psfrag}
\hypersetup{colorlinks}
\zexternaldocument{qkirwan2}
\zexternaldocument{qkirwan3}

\usepackage{color}

\definecolor{darkred}{rgb}{0.5,0,0}
\definecolor{darkgreen}{rgb}{0,0.5,0}
\definecolor{darkblue}{rgb}{0,0,0.5}

\hypersetup{ colorlinks,
linkcolor=darkblue,
filecolor=darkgreen,
urlcolor=darkred,
citecolor=darkblue }

\renewcommand\theenumi{\alph{enumi}}
\renewcommand\theenumii{\roman{enumii}}
\theoremstyle{plain}
\newtheorem{theorem}{Theorem}[section]
\newtheorem{lemma}[theorem]{Lemma}
\theoremstyle{remark}
\newtheorem{remark}[theorem]{Remark}

\newcommand\M{\mathcal{M}}

\renewcommand\M{\mathcal{M}}

\newcommand{\XX}{\mathcal{X}}

\newcommand{\YY}{\mathcal{Y}}
\newcommand{\ZZ}{\mathcal{Z}}

\newcommand{\R}{\mathbb{R}}

\newcommand{\C}{\mathbb{C}}
\newcommand{\cC}{\mathcal{C}}

\newcommand{\Z}{\mathbb{Z}}
\newcommand{\Q}{\mathbb{Q}}

\renewcommand{\P}{\mathbb{P}}
\newcommand{\bA}{\mathbb{A}}

\newcommand\lie[1]{\mathfrak{#1}}

\newcommand{\g}{\lie{g}}

\newcommand{\on}{\operatorname}

\newcommand{\tw}{\on{tw}}
\newcommand{\Ups}{\Upsilon}

\newcommand{\quot}{\on{quot}}

\newcommand{\dual}{\vee}

\newcommand{\Ve}{\on{Vert}}
\newcommand{\Edge}{\on{Edge}}

\newcommand{\Aut}{ \on{Aut} }

\newcommand{\Ad}{ \on{Ad} }

\newcommand{\Hom}{ \on{Hom}}

\newcommand{\Ind}{ \on{Ind}}

\newcommand{\UU}{ \mathfrak{U}}

\newcommand\dirac{/\kern-1.2ex\partial} 
\newcommand\qu{/\kern-.7ex/} 
\newcommand\lqu{\backslash \kern-.7ex \backslash} 

\newcommand\dr{r_+ \kern-.7ex - \kern-.7ex r_-}
 
\usepackage{amsmath, amsfonts,amssymb, latexsym, mathrsfs}
\newtheorem{corollary}[theorem]{Corollary}
\newtheorem{definition}[theorem]{Definition}
\newtheorem{proposition}[theorem]{Proposition}
\newtheorem{example}[theorem]{Example}

\newcommand{\labell}\label

\newcommand{\ovl}{\overline}

\newcommand\eps{\epsilon}

\newcommand{\lan}{\langle}
\newcommand{\ran}{\rangle}

\newcommand{\ti}{\tilde}

\newcommand\pt{\on{pt}}

\newcommand\age{\on{age}}

\newcommand\cF{\mathcal{F}}

\newcommand{\sss}{\on{ss}}

\newcommand\MM{\mathfrak{M}}

\renewcommand\AA{\mathcal{A}}

\newcommand\Sym{\on{Sym}}

\newcommand\ev{\on{ev}}
\newcommand\Eul{\on{Eul}}

\newcommand\mO{\mathcal{O}}

\newcommand\bra[1]{ < \kern-.7ex {#1} \kern-.7ex >} 
\newcommand\bdefn{\begin{definition}}
\newcommand\edefn{\end{definition}}
\newcommand\bea{\begin{eqnarray*}}
\newcommand\eea{\end{eqnarray*}}
\newcommand\bcv{\left[ \begin{array}{r} }
\newcommand\ecv{\end{array} \right] }

\newcommand\bma{\left[ \begin{array} }
\newcommand\ema{\end{array} \right]}
\newcommand\ben{\begin{enumerate}}
\newcommand\een{\end{enumerate}}
\newcommand\beq{\begin{equation}}
\newcommand\eeq{\end{equation}}
\newcommand\bex{\begin{example}}
\newcommand\bsj{\left\{ \begin{array}{rrr} }
\newcommand\esj{\end{array} \right\}}

\newcommand\eex{\end{example}}

\newcommand\sx{*\kern-.5ex_X}

\newcommand{\fr}{{\on{fr}}}

\def\mathunderaccent#1{\let\theaccent#1\mathpalette\putaccentunder}
\def\putaccentunder#1#2{\oalign{$#1#2$\crcr\hidewidth \vbox
to.2ex{\hbox{$#1\theaccent{}$}\vss}\hidewidth}}

\begin{document}

\title[Quantum Kirwan morphism III]{Quantum Kirwan morphism and Gromov-Witten invariants of
  quotients III}

\authors{Chris T. Woodward\thanks{Partially supported by NSF
 grant DMS0904358 and the Simons Center for Geometry and Physics}
\address Department of Mathematics \\ 
Rutgers University \\ 110 Frelinghuysen Road \\ Piscataway, NJ 08854-8019,
U.S.A. 
  \email ctw@math.rutgers.edu
}


\maketitle

\begin{abstract}  
This is the third in a sequence of papers in which we construct a
quantum version of the Kirwan map from the equivariant quantum
cohomology $QH_G(X)$ of a smooth polarized complex projective variety
$X$ with the action of a connected complex reductive group $G$ to the
orbifold quantum cohomology $QH(X \qu G)$ of its geometric invariant
theory quotient $X \qu G$, and prove that it intertwines the genus
zero gauged Gromov-Witten potential of $X$ with the genus zero
Gromov-Witten graph potential of $X \qu G$.  We also give a formula
for a solution to the quantum differential equation on $X \qu G$ in
terms of a localized gauged potential for $X$.  These results overlap
with those of Givental \cite{gi:eq}, Lian-Liu-Yau \cite{lly:mp1},
Iritani \cite{iri:gmt}, Coates-Corti-Iritani-Tseng
\cite{coates:mirrorstacks}, and Ciocan-Fontanine-Kim
\cite{ciocan:genuszerowall}, \cite{ciocan:bigI}.

\end{abstract} 


\tableofcontents

\setcounter{section}{6}
\setcounter{equation}{34}
\setcounter{figure}{13}

\vskip .2in

We continue with the notation in the previous part \cite{qk2}, where
we defined perfect obstruction theories and virtual fundamental
classes for the moduli stacks of stable gauged maps.  In this part we
define the resulting gauged invariants, which come in various flavors
(gauged invariants for fixed scaling, affine gauged invariants, and
invariants with varying scaling).  We show the splitting axioms for
the invariants and deduce the main results of the series: Let a
complex reductive group $G$ act on a smooth polarized projective (or
in some cases, quasiprojective) variety $X$ with only finite
stabilizers on the semistable locus, and let $\Lambda_X^G$ be the
Novikov field for $H^2_G(X,\Q)$.  We construct a quantum Kirwan map
$$\kappa_X^G: QH_G(X) \to QH_{S^1}(X \qu
G)$$ 
and prove the adiabatic limit theorem that the quantum Kirwan map
intertwines the gauged graph potential and graph potential of the git
quotient
$$\tau_X^G: QH_G(X) \to \Lambda_X^G, \quad \tau_{X \qu G}: QH_{S^1}(X \qu
G) \to \Lambda_X^G$$ 
in the limit of large area.  We end with a partial computation of the
quantum Kirwan map in the toric case, that is, when $X$ is a vector
space with a linear action of a torus $G$.

We thank I. Ciocan-Fontanine and B. Kim for pointing out a missing
circle-equivariant term in Example \ref{Ifunction}.

\section{Gauged Gromov-Witten invariants}

\label{alggw}

\subsection{Equivariant Gromov-Witten theory for smooth varieties} 

First we recall the definition equivariant Gromov-Witten invariants
for a smooth projective target using the Behrend-Fantechi machinery
\cite{bf:in}, as explained in Graber-Pandharipande \cite{gr:loc}.  We
adopt the perspective on the splitting axiom adopted in Behrend
\cite{be:gw}: Invariants are defined for any possibly disconnected
combinatorial type, and the splitting axiom can be broken down into
{\em cutting edges} and {\em collapsing edges} axiom.  In preparation
for studying the properties of the virtual fundamental classes,
suppose as in Behrend-Fantechi \cite[p. 51]{bf:in} that there is a
diagram of Deligne-Mumford stacks
$$ \begin{diagram} 
\node{\XX'} \arrow{e,t}{u} \arrow{s,l}{g} \node{\XX} \arrow{s,r}{f} \\
\node{\ZZ'} \arrow{e,t}{v} \node{\ZZ} \end{diagram} $$
where $v: \ZZ' \to \ZZ$ is a local complete intersection morphism with
finite unramified diagonal over a stack $\YY$.  Let $E \to L_\XX$ and
$F \to L_{\XX'}$ be perfect relative obstruction theories for $\XX$
and $\XX'$ over $\YY$, respectively.  A {\em compatibility datum} for
$E$ and $F$ is a triple of morphisms in $D(\mO_{\XX'})$ giving rise to
a morphism of distinguished triangles
$$ \begin{diagram} 
\node{u^*E}  \arrow{e,t}{\phi} \arrow{s} 
\node{F}  \arrow{e,t}{\psi} \arrow{s} 
\node{g^* L_{\ZZ'/\ZZ}}  
\arrow{e,t}{\chi} 
\arrow{s}  
\node{u^* E[1]} \arrow{s} 
  \\
\node{u^* L_{\XX/\YY}}  \arrow{e}
\node{L_{\XX'/\YY}}  \arrow{e} 
\node{L_{\XX'/\XX}} 
 \arrow{e} \node{u^* L_{\XX/\YY}[1]}  
 \end{diagram} 
$$ 
We say that $E,F$ are {\em compatible perfect relative obstruction
  theories} if there exists a compatibility datum.  By
\cite[7.5]{bf:in} if $E,F$ are compatible perfect relative
  obstruction theories, and $\ZZ'$ and $\ZZ$ as above are smooth then
  $ v^![\XX] = [\XX'] .$

\begin{example} \label{cutting} (Cutting an edge for stable maps)
Let $\Ups: \Gamma \to \Gamma'$ be a morphism of graphs disconnecting
an edge, that is, replacing an edge in $\Gamma$ with a pair of
semi-infinite edges in $\Gamma'$.  We have a morphism of stacks of
stable curves $\ovl{\MM}(\Ups): \ovl{\MM}_{g,n,\Gamma} \to
\ovl{\MM}_{g,n+2,\Gamma'} $ obtained by identifying the two additional
markings, and an induced isomorphism $ \ovl{\M}(\Ups):
\ovl{\M}_{g,n,\Gamma} \to \ovl{\M}_{g,n+2,\Gamma'} $, except in the
case that there exists an automorphism of a curve of combinatorial
type $\Gamma'$ interchanging the two markings, in which case it is a
double cover.

The stack of stable maps $\ovl{\M}_{g,n,\Gamma}(X)$ may be identified
(up to a possible automorphism) with the sub-stack of
$\ovl{\M}_{g,n+2,\Gamma'}(X)$ consisting of objects with $u(z_{n+1}) =
u(z_{n+2})$, where $z_{n+1},z_{n+2}$ are the new markings.  That is,
we have a Cartesian diagram
$$ \begin{diagram} \node{\ovl{\M}_{g,n,\Gamma}(X)} \arrow{e} \arrow{s}
  \node{ \ovl{\M}_{g,n+2,\Gamma'}(X)}
  \arrow{s}\\ \node{\ovl{\MM}_{g,n,\Gamma} \times X}
  \arrow{e,t}{\Delta} \node{\ovl{\MM}_{g,n+2,\Gamma'} \times X \times
    X} \end{diagram} $$
where $\Delta$ combines the identification of the moduli stacks with
the diagonal embedding of $X$.  As explained in Behrend
\cite[p.8]{be:gw} for the case of stable maps, the two perfect
relative obstruction theories are compatible which implies
$$[\ovl{\M}_{g,n,\Gamma}(X)] = \Delta^!  [\ovl{\M}_{g,n+2,\Gamma'}(X)]
.$$
Indeed if $\Gamma'$ is obtained from $\Gamma$ by cutting an edge then
we check that the obstruction theories are compatible over $\Delta$.
Consider the Cartesian diagram
$$ \begin{diagram} \node{\ovl{\M}_{g,n,\Gamma}(X)}
\arrow{s,l}{\Psi} \arrow{e,t}{\ovl{\M}(\Ups,X)} 
\node{\ovl{\M}_{g,n+2,\Gamma'}(X)} 
 \arrow{s} \\
 \node{\ovl{\MM}_{g,n,\Gamma}
\times X} \arrow{e} 
 \node{\ovl{\MM}_{g,n+2,\Gamma'} 
 \times X \times X}
\end{diagram} $$
Let $\pi: \cC \to \ovl{\M}_{g,n,\Gamma}(X)$ denote the universal
curve, and let $\cC'' = \ovl{\M}(\Ups,X)^* \cC'$ be the curve over
$\ovl{\M}_{g,n+2,\Gamma'}(X)$ obtained by normalizing at the node
corresponding to the edge, with $p: \cC'' \to \cC$ the projection, and
$ \ev'': \cC'' \to X, \quad \ev: \cC \to X $ 
the universal maps.  So $\cC$ is obtained from $\cC''$ by identifying
the two sections $x_1,x_2$ of $\cC''$, and is equipped with a section
$x$ induced from $x_1,x_2$.  We have a short exact sequence of
complexes relating the push-forward on $\cC,\cC''$,
$$ 0 \to \ev^* TX \to p_* p^* \ev^* TX \to x_* x^* \ev^* TX \to 0 , $$
and so an exact triangle
\begin{equation} \label{doubleprime} 
 R \pi_* \ev^* TX \to R \pi_*'' p^* \ev^* TX \to x^* \ev^* TX \to R
 \pi_* \ev^* TX [1] .\end{equation}
Now if $ E_\Gamma(X) := ( R\pi_* \ev^{*} TX)^\dual $ then
$$ \ovl{\MM}(\Ups)^* E_{\Gamma'}(X) = (R \pi_*'' \ev^{'',*} TX)^\dual
= (R \pi_*'' p^* \ev^* TX)^\dual .$$
Moreover we have an exact triangle
$$ \Psi^* L_\Delta [-1] \to \ovl{\M}(\Ups,X)^* E_{\Gamma''}(X) \to
E_\Gamma(X) \to \Psi^* L_\Delta .$$
This gives rise to a homomorphism of distinguished triangles
$$ \tiny \begin{diagram} 
\node{\ovl{\MM}(\Ups)^* E_{\Gamma'}(X)} \arrow{e} \arrow{s}  
\node{E_\Gamma(X)}
\arrow{e} \arrow{s} \node{\Psi^* L_\Delta} 
\arrow{s} 
\\
\node{\ovl{\MM}(\Ups)^* L_{\ovl{\M}_{g,n+2,\Gamma'}(X)/\ovl{\MM}_{g,n+2,\Gamma'}}} \arrow{e} 
\node{L_{\ovl{\M}_{g,n,\Gamma}(X)/\ovl{\MM}_{g,n,\Gamma}}} \arrow{e} \node{L_{\ovl{\M}(\Ups,X)}.} 
 \end{diagram} $$
\end{example} 

\begin{example}\label{collapsing} (Collapsing an edge for stable maps)
 Let $\Ups: \Gamma' \to \Gamma$ be a morphism of modular graphs given
 by collapsing an edge.  Associated to $\Ups$ are morphisms of Artin
 resp. Deligne-Mumford stacks
$$ \ovl{\MM}(\Ups): \ovl{\MM}_{g,n,\Gamma'} \to \ovl{\MM}_{g,n,\Gamma}, \quad \ovl{\M}(\Ups):
\ovl{\M}_{g,n,\Gamma'} \to \ovl{\M}_{g,n,\Gamma} .$$
The inclusion of $\ovl{\MM}_{g,n,\Gamma}(X)$ to
$\ovl{\MM}_{g,n,\Gamma'}(X)$ induces an {\em isomorphism} of perfect
relative obstruction theories.  As in Behrend \cite{be:gw}, the
relative obstruction theories for $\ovl{\M}_{g,n,\Gamma}(X),
\ovl{\M}_{g,n,\Gamma'}(X)$ are related by pull-back.  Letting
$s(\Gamma), s(\Gamma')$ denote the stabilizations of $\Gamma,\Gamma'$
consider the diagram from \cite[p. 15]{be:gw}
$$ \begin{diagram} \node{ \sqcup_{d' \to d}
    \ovl{\M}_{g,n,\Gamma'}(X,d')} \arrow{e} \arrow{s}
  \node{\ovl{\M}_{g,n,s(\Gamma')} \times_{\ovl{\M}_{g,n,s(\Gamma)}}
    \ovl{\M}_{g,n,\Gamma}(X,d)} \arrow{s} \arrow{e}
  \node{\ovl{\M}_{g,n,\Gamma}(X,d)} \arrow{s} \\ \node{
    \ovl{\MM}_{g,n,\Gamma'}} \arrow{e} \arrow{se} \node{
    \ovl{\M}_{g,n,s(\Gamma')} \times_{\ovl{\M}_{g,n,s(\Gamma)}}
    \ovl{\MM}_{g,n,\Gamma}} \arrow{s} \arrow{e}
  \node{\ovl{\MM}_{g,n,\Gamma}} \arrow{s} \\ \node{}
  \node{\ovl{\M}_{g,n,s(\Gamma')}} \arrow{e} \node{\ovl{\M}_{g,n,s(\Gamma)}.}
\end{diagram} 
$$ 
All the squares are Cartesian and it follows as in \cite{be:gw} (see
especially \cite[Proposition 8]{be:gw}, which uses bivariant Chow
theory for representable morphisms of Artin stacks) that
$$ \ovl{\M}(\Ups)^!  [\ovl{\M}_{g,n,\Gamma}(X,d)] =  \ovl{\M}(\Ups,X)_* \sum_{d'
  \mapsto d} [\ovl{\M}_{g,n,\Gamma'}(X,d')] $$
where 
$$\ovl{\M}(\Ups,X): \ovl{\M}_{g,n,\Gamma'}(X,d') \to
\ovl{\M}_{g,n,\Gamma'} \times_{\ovl{\M}_{g,n,\Gamma}}
\ovl{\M}_{g,n,\Gamma}(X,d) $$
is the identification with the fiber product. 
\end{example}  

It follows that the virtual fundamental classes
$[\ovl{\M}_{g,n,\Gamma}(X,d)] \in A^G(\ovl{\M}_{g,n,\Gamma}(X,d))$
satisfy the following properties as in Behrend \cite{be:gw}:
\begin{proposition} \label{vfcprops}
\begin{enumerate} 
\item {\rm (Constant maps)} If $ d =0 $ then
  $[\ovl{\M}_{g,n,\Gamma}(X,d)]$ is obtained by cap product of $[X
  \times \ovl{\M}_{g,n,\Gamma}]$ with the Euler class of $R^1 p_* e^* TX$.
\item {\rm (Cutting edges)} If $\Ups: \Gamma \to \Gamma'$ is a morphism of
  modular graphs of type cutting an edge then
$ [\ovl{\M}_{g,n,\Gamma}(X,d)] = \Delta^! [\ovl{\M}_{g,n+2,\Gamma'}(X,d')] $
where $\Delta: X \to X \times X $ is the diagonal.
\item {\rm (Collapsing edges)} If $\Ups: \Gamma \to \Gamma'$ is a
  morphism of graphs of type collapsing an edge then
$$ \M(\Ups)^! [\ovl{\M}_{g,n,\Gamma'}(X,d')] = \ovl{\cF}(\Ups,X)_*
  \sum_{d \mapsto d'} [\ovl{\M}_{g,n,\Gamma}(X,d)] $$
where 
$$\ovl{\cF}(\Ups,X): \ovl{\M}_{g,n,\Gamma}(X,d) \to
\ovl{\M}_{g,n,s(\Gamma)} \times_{\ovl{\M}_{g,n,s(\Gamma')}}
\ovl{\M}_{g,n,\Gamma'}(X,d') .$$
\item {\rm (Forgetting Tails)} If $\Ups: \Gamma \to \Gamma'$ is a
  morphism of graphs of type forgetting a tail then 
$$ \ovl{\M}(\Ups,X)^! [\ovl{\M}_{g,n,\Gamma'}(X,d)] =
  [\ovl{\M}_{g,n+1,\Gamma}(X,d)] $$
where $\ovl{\M}(\Ups,X)$ was defined in \cite[Example 4.3]{qk2}.
\end{enumerate} 
\end{proposition} 

We now pass from Chow groups/rings to homology/cohomology with
rational coefficients.  (One can work with more general theories here,
as in Behrend-Manin \cite{bm:gw}.)  For any cohomology classes $\alpha
\in H_G(X,\Q)^n$ and $\beta \in H(\ovl{M}_{g,n,\Gamma},\Q)$ (if $2g +
n \ge 3$) pairing with the virtual fundamental class
$[\ovl{\M}_{g,n,\Gamma}(X,d)] \in H(\ovl{M}_{g,n,\Gamma}(X,d)) $
defines a {\em Gromov-Witten invariant}
$$ \lan \alpha ; \beta \ran_{\Gamma,d} =
\int_{[\ovl{\M}_{g,n,\Gamma}(X,d)]} \ev^* \alpha \cup f^* \beta \in 
H(BG).$$
These invariants satisfy axioms for morphisms of modular graphs:
\begin{proposition}   \label{gwprops} 
\begin{enumerate} 
\item {\rm (Cutting edges)}  If $\Gamma'$ is obtained from $\Gamma$ by cutting an edge then
$$ \lan \alpha ; \beta \ran_{\Gamma,d} = \sum_{i=1}^N \lan \alpha,
  \delta_i,  \delta^i ; \ovl{\M}(\Ups)^* \beta \ran_{\Gamma',d}$$
where $(\delta_i)_{i=1}^N,(\delta^i)_{i=1}^N$ are dual bases for
$H_G(X)$ over $H(BG)$.
\item {\rm (Collapsing edges)} If $\Ups: \Gamma \to \Gamma'$ is a
  morphism collapsing an edge then
$$ 
\lan \alpha ; \beta \cup \gamma 
\ran_{\Gamma',d'} =
  \sum_{d \mapsto d'} \lan \alpha ; 
\ovl{\M}(\Ups)^* 
\beta \ran_{\Gamma,d}$$
where $\gamma \in H^2 (\ovl{\M}_{g,n,\Gamma'})$ is the dual
class for $\ovl{\M}(\Ups): \ovl{\M}_{g,n,\Gamma} \to
\ovl{\M}_{g,n,\Gamma'}$.
\item {\rm (Forgetting tails)} If $\Ups: \Gamma \to \Gamma'$ is a
  morphism forgetting a tail then for $\alpha' \in H^2_G(X)$,
$$ \lan \alpha, \alpha' ; \ovl{\M}(\Ups)^* \beta \ran_{\Gamma,d} =
  (d, \alpha') \lan \alpha ; \beta \ran_{\Gamma',d'}$$
\end{enumerate} 
\end{proposition} 

\begin{proof} These follow from Proposition \ref{vfcprops} as in 
(the more abstract formulation) in 
Behrend-Manin \cite[Theorem
    9.2]{bm:gw} to which we refer the reader for more detail.
\end{proof} 

\begin{definition} {\rm (Novikov field)} 
The Novikov field $\Lambda_X$ for $X$ is the set of all maps $a:
H_2(X):=H_2(X,\Z)/\on{torsion} \to \Q$ such that for every constant
$c$, the set of classes
%
$$ \left\{ d \in H_2(X,\Z)/\on{torsion}, \ \lan [\omega], d \ran \leq
c, \ \ a(d) \neq 0 \right\} $$
on which $a$ is non-vanishing is finite.  The delta function at $d$ is
denoted $q^d$.  Addition is defined in the usual way and
multiplication is convolution, so that $q^{d_1} q^{d_2} = q^{d_1 +
  d_2}$.  
\end{definition} 

Define as vector space the {\em quantum cohomology of $X$}
$$ QH_G(X) := H_G(X,\Q) \otimes \Lambda_X .$$
Define genus $g$ correlators
$$\lan \cdot , \cdot \ran_{g,n} = \sum_{d \in H_2(X)} q^d \lan \cdot ,
  \cdot \ran_{\Gamma,d} $$
where $\Gamma$ is a genus $g$ graph with one vertex and $n$
semi-infinite edges.  By Proposition \ref{gwprops},

\begin{theorem} (\cite{gi:eq}, \cite{be:gw})
After tensoring with the field of fractions of $H(BG)$ the space
$QH_G(X)$ equipped with the maps $\lan \cdot, \cdot \ran_{g,n}$ forms
a cohomological field theory.
\end{theorem} 

Restricting to genus zero we obtain a CohFT algebra: Maps
$$ \mu^n: QH_G(X)^n \times H(\ovl{M}_{0,n+1},\Q) \to QH_G(X) $$
defined by
$$ (\mu^n(\alpha_1,\ldots,\alpha_n;\beta),\alpha_0) = \sum_{d \in
  H_2(X,\Z)} q^d \lan \alpha_0,\ldots,\alpha_n ; \beta \ran_{0,d} \in
\Lambda_X .$$
Here $( \ , \ )$ denotes the pairing on $QH_G(X)$ induced by cup
product and integration over $H_G(X)$. 

A related collection of invariants is expressed as the integrals over
{\em parametrized} stable maps to $X$.  Let $\Hom(C,X,d) \subset
\Hom(C,X)$ denote the subscheme of maps of class $d \in H_2(X,\Z)$.
Compactifications of $\Hom(C,X,d)$ are provided by so-called {\em
  graph spaces}
$$ \ovl{\M}_n(C,X,d) := \ovl{\M}_{g,n}(C \times X,(1,d))$$ 
of stable maps $u: \hat{C} \to C \times X$ of degree $(1,d)$.  Each
stable map $u = (u_C,u_X): \hat{C} \to C \times X$ has a single
component $\hat{C}_0 \subset \hat{C}$ that maps isomorphically onto
$C$ via $u_C$, with all other components mapping to points.  We denote
by
\begin{equation} \label{evC}
 \ev: \ovl{\M}_n(C,X) \to X^n, \quad \ev_C: \ovl{\M}_n(C,X) \to C^n \end{equation}
the evaluation maps followed by projection on the second, resp. first
factor.  The stacks $\ovl{\M}_n(C,X,d)$ have equivariant relatively
perfect obstruction theories over $\ovl{\MM}_n(C)$ with complex given by
$ Rp_* e^* TX$, where $p: \ovl{\cC}_n(C,X) \to \ovl{\M}_n(C,X)$ is the
universal curve and $e: \ovl{\cC}_n(C,X) \to C$ the evaluation map.
For any cohomology classes $\alpha \in H_G(X,\Q)^n$ and $\beta \in
H(\ovl{\M}_{n,\Gamma}(C),\Q)$ pairing with the virtual fundamental
class $[\ovl{\M}_{n,\Gamma}(C,X,d)] \in H(\ovl{\M}_{n,\Gamma}(C,X,d)) $
defines a {\em graph Gromov-Witten invariant}
\begin{equation} \label{graphinv}
 \lan \alpha ; \beta \ran_{C,\Gamma,d} =
\int_{[\ovl{\M}_{n,\Gamma}(C,X,d)]} \ev^* \alpha \cup f^* \beta \in
H(BG) .\end{equation} 
%
These invariants satisfy axioms for morphisms of {\em rooted} modular
trees, similar to those in Proposition \ref{gwprops} which we omit to
save space.
%
%
Define 
\begin{multline} 
 \tau_X^{n}: QH_G(X)^n \times H(\ovl{M}_n(C)) \to \Lambda_X^G \otimes
 H(BG), \ (\alpha,\beta) \mapsto \sum_{d} q^d \lan \alpha ; \beta
 \ran_{C,d} .\end{multline}
%
The splitting axiom for these invariants implies:

\begin{theorem}  \label{tracethm}
The maps $(\tau_{X}^{n})_{n \ge 0}$ define a CohFT trace on the CohFT
algebra $QH_G(X)$.
\end{theorem} 

The Gromov-Witten invariants have a natural circle-invariant extension
given by interpreting the equivariant parameter as the first Chern
class of the tangent lines at the markings.  Let
$\psi_i \in H^2(\ovl{\M}_n^G(C,X))$ denote the Chern class of the
cotangent line at the $i$-th marking.  Given 
$\alpha_i(\zeta) \in QH_{G \times S^1}(X) \cong QH_G(X)[\zeta]$ for $I = 1,\ldots, n$ 
define 
\begin{multline}
 \lan \alpha(\zeta) ; \beta \ran_{\Gamma,d} =
\int_{[\ovl{\M}_{g,n,\Gamma}(X,d)]} \ev_1^* \alpha_1 |_{\zeta =
  -\psi_1} \\
\ldots \ev_n^* \alpha_n |_{\zeta = -\psi_n}
 \cup f^* \beta \in
H(BG).\end{multline}
This circle-invariant extension will play a role in the adiabatic
limit theorem given later.


\subsection{Gromov-Witten theory for smooth Deligne-Mumford stacks} 

We review the orbifold Gromov-Witten theory developed by Chen-Ruan
\cite{cr:orb} and Abramovich-Graber-Vistoli \cite{agv:gw}, needed in
our case if the geometric invariant theory quotient $X \qu G$ is an
orbifold.  For simplicity, we restrict to the case without group
action.

Let $\XX$ be a proper smooth Deligne-Mumford stack.  The moduli stack
of twisted stable maps admits evaluation maps
$$ \ev: \ovl{\M}_{g,n}(\XX) \to \ovl{I}_\XX^n, 
\quad  \ovl{\ev}: \ovl{\M}_{g,n}(\XX) \to \ovl{I}_\XX^n, 
 $$
where the second is obtained by composing with the involution of the
rigidified inertia stack $\ovl{I}_{\XX} \to \ovl{I}_{\XX} $ induced by
the automorphism of the group $\mu_r$ of $r$-th roots of unity $\mu_r
\to \mu_r, \varphi \mapsto \varphi^{-1}$.  (See \cite[Section
  4.3]{qk2} for the definition of the rigidified inertia stack.)
The virtual fundamental classes satisfy splitting axioms for morphisms
of modular graphs, in particular, for cutting an edge in which case
one of the evaluation maps is taken to be with respect to opposite
signs on the pair of marked points created by the cutting.  Given a
homology class $d \in H_2(\XX,\Q)$ and non-negative integers $g,n$,
let $\ovl{\M}_{g,n}(\XX,d)$ denote the moduli stack of stable maps to
$\XX$ with class $d$.  The virtual fundamental classes
$[\ovl{\M}_{g,n,\Gamma}(\XX,d)]$ satisfy the splitting axioms for
morphisms of modular graphs similar to those in the case that $X$ is a
variety.  Orbifold Gromov-Witten invariants are defined by virtual
integration of pull-back classes using the evaluation maps above. For
non-negative integers $n_- + n_+ = n$ denote by $\ev^*_{n_+}$ resp.  $
\ovl{\ev}^*_{n_-}$ the untwisted resp. twisted evaluation map on the
first $n_+$ resp. last $n_-$ markings.  Define {\em Gromov-Witten
  invariants}
\begin{multline}
 H(\ovl{I}_\XX)^{n_+} \times H(\ovl{I}_\XX)^{n_-} \times H(\ovl{M}_{g,n})
 \to \Q, \\ (\alpha_+,\alpha_-, \beta) \mapsto \lan
 \alpha_+,\alpha_-,\beta \ran_{\Gamma,d} = \int_{
   [\ovl{\M}_{g,n,\Gamma}(\XX,d)]} \ev^*_{n_+} \alpha_+ \cup \ovl{\ev}^*_{n_-}
 \ovl{\alpha}_- \cup f^* \beta.\end{multline}
The orbifold Gromov-Witten invariants satisfy properties similar to
those for usual Gromov-Witten invariants, with the notable exception
\cite[6.1.4]{agv:gw} that if $\Gamma'$ is obtained from $\Gamma$ by
cutting an edge then
\begin{equation} \label{orbsplit} \langle \alpha ; \beta \rangle_{\Gamma,d} = \sum_k \langle
\alpha, \delta_k , \delta^k; \ovl{\M}(\Ups)^* \beta
\rangle_{\Gamma',d} \end{equation} 
where $\delta_k,\delta^k$ are dual bases of $H(\ovl{I}_\XX)$ with
respect to a different inner product: the inner product defined using
re-scaled integration $( [\ovl{I}_{\XX}], r \cdot)$ where $r:
\ovl{I}_{\XX} \to \Z_{ \ge 0}$ is the order of the isotropy group.
The definition of orbifold Gromov-Witten invariants leads to the
definition of orbifold quantum cohomology as follows.

\begin{definition} {\rm (Orbifold quantum cohomology)}
To each component $\XX_i$ of $\ovl{I}_\XX$ is assigned a rational
number $\age(\XX_i)$ as follows.  Let $(x,g)$ be an object of $\XX_i$.
The element $g$ acts on $T_x \XX$ with eigenvalues $\alpha_1,
\ldots,\alpha_n$ with $ n = \dim(\XX)$.  Let $r$ be the order of $g$
and define $s_j \in \{ 0,\ldots, r - 1 \}$ by $\alpha_j = \exp( 2\pi i
s_j / r)$.  The {\em age} is defined by
$$ \age(\XX_i)  = (1/r) \sum_{j=1}^n s_j .$$
Let $ \Lambda_\XX \subset \Hom(H_2(X,\Q),\Q) $ denote the Novikov
field of linear combinations of formal symbols $q^d, d \in
H_2(\XX,\Q)$ where for each $c$, only finitely many $q^d$ with
$(d,[\omega]) < c$ have non-zero coefficient.  Let
$$ QH(\XX) = H(\ovl{I}_\XX) \otimes \Lambda_\XX $$
denote the {\em orbifold quantum cohomology} equipped with the {\em
  age grading}
$$ QH^\bullet(\XX) = \bigoplus_{\XX_i \subset \ovl{I}_\XX} H^{\bullet - 2
  \age(\XX_i)}(\XX_i) \otimes \Lambda_\XX .$$
\end{definition}

\begin{theorem}  The orbifold Gromov-Witten invariants
define the structure of a CohFT on $QH(\XX)$, in particular, a CohFT
algebra structure on $QH(\XX)$ and the graph invariants define a trace
on $QH(\XX)$.
\end{theorem}  

\begin{proof}   This follows from the splitting axiom \eqref{orbsplit},
and the analogous splitting axiom for the graph invariants whose proof
is similar.  
\end{proof} 


\subsection{Twisted Gromov-Witten invariants} 

We also describe twisted versions of Gromov-Witten invariants arising
from vector bundles on the target, see for example Coates-Givental
\cite{co:qrr}.  Under suitable positivity assumptions, these
invariants are equal to the Gromov-Witten invariants of hypersurfaces
defined by sections.  

\begin{definition} {\rm (Twisting class and twisted Gromov-Witten invariants)} 
Let $E$ be a $G$-equivariant complex vector bundle over a smooth
projective $G$-variety $X$.  Pull-back under the evaluation map $e:
\ovl{\cC}_{g,n}(X) \to X$ on the universal curve gives rise to a vector
bundle $\ev^{*} E \to \ovl{\cC}_{g,n}(X)$, which we can push down to an
{\em index}
$$\Ind_G(E) := Rp_* e^* E$$
in the derived category of bounded complexes of equivariant coherent
sheaves on $\ovl{\M}_{g,n}(X)$.  Since $p$ is a local complete
intersection morphism, $\Ind_G(E)$ admits a resolution by vector
bundles, see \cite[Appendix]{co:qrr}, and we may consider the
(invertible) equivariant Euler class
$$ \eps(E) := \Eul_{G \times \C^\times}(\Ind_{G}(E)) \in
H_G(\ovl{\M}_{g,n}(X)) \otimes \Q[\varphi,\varphi^{-1}]$$
where $\varphi$ is the parameter for the action of $\C^\times$ by scalar
multiplication in the fibers.  The {\em twisted equivariant
  Gromov-Witten invariants} associated to $E \to X$ and type $\Gamma$
are the maps
\begin{multline} H_G(X)^n \times H(\ovl{M}_{g,n,\Gamma}) \to  H_G(\pt,\Q) \otimes
 \Q[\varphi,\varphi^{-1}]
\\
 \lan \alpha ; \beta \ran_{\Gamma,E,d} =
 \int_{[\ovl{\M}_{n,\Gamma}(C,X)]} \ev^* \alpha \cup f^* \beta \cup
 \eps(E) 
.\end{multline}
\end{definition} 

\begin{proposition} 
The twisted invariants satisfy the properties:
\begin{enumerate} 
\item {\rm (Collapsing edges)} If $\Ups: \Gamma \to \Gamma'$ is of
  type collapsing an edge then for any labelling $d'$ of $\Gamma'$,
$$ \langle  \alpha ; \beta \cup  \gamma \rangle_{\Gamma',d',E}  
= \sum_{d \mapsto d'} 
\langle  \alpha ; \ovl{\M}(\Ups)^* \beta  \rangle_{\Gamma,d,E}   $$
where $\gamma$ is the dual class to $\ovl{\M}(\Ups)$.
\item {\rm (Cutting edges)} If $\Ups: \Gamma \to \Gamma'$ is of type
  cutting an edge then
$$ \langle \alpha ; \beta \rangle_{\Gamma,d,E} = \sum_k \langle
  \alpha, \delta_k \cup \Eul_{G \times \C^\times}(E), \delta^k; \ovl{\M}(\Ups)^*
  \beta \rangle_{\Gamma',d,E} .$$
\item {\rm (Forgetting tails)} If $\Ups: \Gamma \to \Gamma'$ is a
  morphism forgetting a tail, which corresponds to the last marking
  $z_n$  then for $\alpha' \in H^2_G(X), \alpha \in H_G(X)^{n-1}$,
$$ \lan \alpha, \alpha' ; \ovl{\M}(\Ups)^* \beta \ran_{\Gamma,d,E} =
  (d, \alpha') \lan \alpha ; \beta \ran_{\Gamma',d',E}$$
\end{enumerate} 
\end{proposition} 

\begin{proof}  We discuss only the cutting-edge axiom; the rest
are similar to those in the untwisted case.  Let $p: \cC' \to \cC$ be
the normalization and $x$ the section given by the node corresponding
to the cut edge.  The short exact sequence of sheaves
$$ 0 \to E \to p^* p_* E \to x^* x_* E \to 0 $$
gives rise to an exact triangle in the derived category of bounded
complexes of coherent sheaves
$$ R \pi_* \ev^* E \to R \pi_*'' p^* \ev^* E \to x^* \ev^* E \to R
\pi_* \ev^* E [1] .$$
Taking Euler classes gives the result.
\end{proof} 

A receptacle for twisted Gromov-Witten invariants is the equivariant
cohomology of a point with the equivariant parameter inverted.  This
means that twisted composition maps take values in the equivariant
cohomology of the space with equivariant parameter inverted. Define
$$QH_{G \times \C^\times}(X,\Q) = QH_G(X,\Q) \otimes \Q[\varphi,\varphi^{-1}]$$ 
and define twisted composition maps
$$ \mu_{E}^{g,n}: QH_{G \times \C^\times}(X,\Q)^n \times H(\ovl{\M}_{g,n+1},\Q)
\to QH_{G \times \C^\times}(X,\Q), $$
$$ (\mu^{g,n}_{E}(\alpha_1,\ldots,\alpha_n;\beta),\alpha_0) :=
\sum_{d \in H_2(X)} q^d \lan \alpha_0 \cup \Eul_{G \times
  \C^\times}(E),\ldots,\alpha_n ; \beta \ran_{\Gamma,d,E} $$
where $\Gamma$ is a genus $g$ graph with a single vertex.  Discussion
of twisted composition maps can be found in e.g. Pandharipande
\cite{pa:ra}.

\begin{theorem}  \label{ewg2}   {\rm (Equivariant twisted Gromov-Witten invariants
define a CohFT algebra)} Suppose that $X$ is a smooth projective
  $G$-variety  and $E \to X$ is a $G$-equivariant vector bundle.
  The datum $(QH_{G \times \C^\times}(X,\Q), (\mu^{g,n}_{E})_{g,n \ge
    0} )$ form a CohFT algebra, denoted $QH_G(X,E)$.
\end{theorem}  


\label{algggw}

\subsection{Gauged Gromov-Witten invariants} 

In this section we define gauged Gromov-Witten invariants.  As in
Behrend \cite{be:gw}, invariants are defined for any possibly
disconnected combinatorial type, and the splitting axiom can be broken
down into {\em cutting edges} and {\em collapsing edges} axiom.
However, the definition for disconnected type requires an additional
datum, of an assignment of each non-root component to a semi-infinite
edge of a root component.  The equivariant virtual classes for the
non-root components combine with the non-equivariant virtual classes
for the root component to a non-equivariant virtual class for moduli
space for disconnected type.

\begin{definition}  {\rm (Virtual fundamental classes for moduli stacks of gauged maps)} 
Let $X$ be a smooth projective $G$-variety. 
\begin{enumerate} 
\item {\rm (Combinatorial type with a single vertex)} We already
  remarked in \cite[Example 6.6]{qk2}
that if $\ovl{\M}_n^G(C,X)$ is a Deligne-Mumford stack, then it has a
perfect obstruction theory, given by the dual of the derived
push-forward of the pull-back of the tangent complex $ (Rp_* e^*
T(X/G))^\dual$ where $p: \ovl{\cC}_n^G(X,d) \to \ovl{\M}_n^G(X,d)$ is
the universal curve, $e: \ovl{\cC}_n^G(X,d) \to X/G$ the universal
stable gauged map, and $T(X/G)$ the tangent complex to $X/G$.  Hence
one obtains a virtual fundamental class of expected dimension
$$[\ovl{\M}_n^G(C,X,d)] \in A(\ovl{\M}_n^G(C,X,d)) .$$
\item {\rm (Connected combinatorial type)} More generally, given any
  connected rooted tree $\Gamma$ we denote by
  $\ovl{\M}^G_{n,\Gamma}(C,X,d)$
  resp. $\ovl{\M}^{G,\fr}_{n,\Gamma}(C,X,d)$ the moduli stack of stable
  gauged maps resp. with framings at the markings of combinatorial
  type $\Gamma$ and class $d$.  Under the assumption that
  $\ovl{\M}_n^G(C,X,d)$ is Deligne-Mumford, the action of $G^n$ on
  $\ovl{\M}^{G,\fr}_{n,\Gamma}(C,X,d)$ is locally free.  The same
  construction gives a virtual fundamental class
$$ [\ovl{\M}^G_{n,\Gamma}(C,X,d)] \in A(\ovl{\M}^G_{n,\Gamma}(C,X,d))
\cong A^{G^n}(\ovl{\M}^{G,\fr}_{n,\Gamma}(C,X,d)).$$
\item {\rm (Disconnected combinatorial type)} Suppose $\Gamma =
  \Gamma_0 \cup \ldots \cup \Gamma_l$ with $\Gamma_0$ containing the
  root vertex is given the {\em additional datum} of a map from the
  non-root components to the root edges: Suppose that $\Gamma =
  \Gamma_0 \cup \Gamma_1 \cup \dots \cup \Gamma_l$ is a disconnected
  rooted $H_2^G(X)$-labelled graph such that $\Gamma_j$ has
  semi-infinite edges $I_j$, and for each $ j = 1,\dots, l$ is given a
  semi-infinite edge $e(j)$ of $\Gamma_0$.  We denote by
$$ \ovl{\M}^{G,f}_{n,\Gamma}(C,X,d) = 
( \ovl{\M}^{G,\fr}_{n_0,\Gamma_0}(C,X,d_0) \times \prod_{j=1}^l 
  \ovl{\M}_{0,n_j,\Gamma_j}(X,d_j)) / G^{n_0} $$
where the action of the $i$-th factor in $G^{n_0}$ acts at the $i$-th
framing on the principal component, and diagonally on the components
corresponding to $\Gamma_j$ with $e(j) = i$.  We have virtual
fundamental classes
$$ [\ovl{\M}^G_{n_0,\Gamma_0}(C,X,d_0)] \in
A(\ovl{\M}^G_{n_0,\Gamma_0}(C,X,d)) \cong
A_{G^{n_0}}(\ovl{\M}^{G,\fr}_{n_0,\Gamma_0}(C,X,d)) $$
$$
[\ovl{\M}_{0,n_j,\Gamma_j}(X,d_j)] \in
A_G(\ovl{\M}_{0,n_j,\Gamma_j}(X,d_j)) $$
and so a virtual fundamental class 
$$ [ \ovl{\M}^{G}_{n,\Gamma}(C,X,d)] = \cup_{d = d_0 + \ldots + d_l} [
  \ovl{\M}^{G,\fr}_{n_0,\Gamma_0}(C,X,d_0)] \times \prod_{j=1}^l [
  \ovl{\M}_{0,n_j,\Gamma_j}(X,d_j)] $$
in 
$$ A_{G^{n_0}}(\ovl{\M}^{G,\fr}_{n_0,\Gamma_0}(C,X,d) \times \prod_j
\ovl{\M}_{0,n_j}(X,d_j)) \cong A(\ovl{\M}^{G,f}_{n,\Gamma}(C,X,d)) .$$
\end{enumerate}
\end{definition}

These classes satisfy the following properties similar to those in
\cite{be:gw}:

\begin{proposition} \label{gaugedvfcprops}
\begin{enumerate} 
\label{constantproducts}
\item {\rm (Constant maps)} If $ d =0 $ and $\on{genus}(C) = 0 $ then
  $\ovl{\M}^G_{\Gamma,n}(C,X,d) = (X \qu G) \times
  \ovl{\M}_{\Gamma,n}(C)$ and $[\ovl{\M}^G_{n,\Gamma}(C,X,d)] = [X \qu G
    \times \ovl{\M}_{n,\Gamma}(C)]$.
\item {\rm (Cutting edges)} If $\Gamma'$ is obtained from $\Gamma$ by
  cutting an edge then (with the obvious labelling of the additional
  component)
$ [\ovl{\M}^G_{n,\Gamma}(C,X,d)] = \Delta^!
  [\ovl{\M}^{G}_{n+2,\Gamma'}(C,X,d)]. $
\item {\rm (Collapsing edges)} If $\Ups: \Gamma' \to \Gamma$ is a
  morphism collapsing an edge then $ \ovl{\M}(\Ups)^!
  [\ovl{\M}^G_{n,\Gamma}(C,X,d)]$ is the push-forward of $\sum_{d'
    \mapsto d} [\ovl{\M}^G_{n,\Gamma'}(C,X,d')] $ under 
$$\ovl{\cF}(\Ups,X): \ovl{\M}^G_{n,\Gamma'}(C,X,d') \to
  \ovl{\M}_{n,s(\Gamma')}(C) \times_{\ovl{\M}_{n,s(\Gamma)}(C)}
  \ovl{\M}^G_{n,\Gamma}(C,X,d) .$$
\item {\rm (Forgetting tails)} If $\Ups: \Gamma \to \Gamma'$ is a
  morphism forgetting a tail then
$$ \ovl{\M}(\Ups)^! [\ovl{\M}_{n,\Gamma'}^G(C,X,d)] =
  [\ovl{\M}^G_{n+1,\Gamma}(C,X,d)] .$$
\end{enumerate} 
\end{proposition} 

\begin{proof}  The items  (a), (b) and (d) are similar to the 
ordinary Gromov-Witten case considered in Behrend \cite{be:gw} and
left to the reader.  For a morphism $\Ups$ cutting an edge for gauged
maps, recall from \cite[Proposition 5.21]{qk2}
that
$\ovl{\M}^{G}_{n,\Gamma}(C,X)$ may be identified with the fiber product
$\ovl{\M}^{G}_{n,\Gamma'}(C,X) \times_{(X/G)^2} (X/G)$ over the
diagonal $\Delta: (X/G) \to (X/G)^2$.  We denote by
$$\ovl{\M}(\Upsilon,X): \ovl{\M}^{G}_{n,\Gamma}(C,X) \to
\ovl{\M}^{G}_{n,\Gamma'}(C,X)$$ 
the resulting morphism.  We check that the obstruction theories
$E_{\Gamma'}$ and $E_{\Gamma}$ are compatible over $\Delta$.  Let
$\cC$ denote the universal curve over $\ovl{\M}^{G}_{n,\Gamma}(C,X)$,
similarly for $\cC'$ and $\Gamma'$.  Let $\cC'' = \M(\Upsilon,X)^*
\cC'$ be the curve over $\ovl{\M}^{G}_{n,\Gamma'}(C,X)$ obtained by
normalizing at the node corresponding to the edge, with $f: \cC'' \to
\cC$ the projection and $ e'': \cC'' \to X/G, \quad e: \cC \to X/G $
the universal maps.  So $\cC$ is obtained from $\cC''$ by identifying
the two sections $x_1,x_2$ of $\cC''$, and is equipped with a section
$x$ induced from $x_1,x_2$.  The short exact sequence of complexes of
coherent sheaves
$$ 0 \to e^*
T_{X/G} \to f_* f^* e^* T_{X/G} \to x_* x^* e^* T_{X/G} \to 0 $$
(viewing $T_{X/G}$ as a two-term complex) induces an exact triangle in
the derived category
$$ R p_* e^* T_{X/G} \to R p''_* f^* e^* T_{X/G} \to x^* e^* T_{X/G} \to
R p_* e^* T_{X/G} [1] .$$
We have relative obstruction theories with complexes
$$ E_\Gamma := ( Rp_* e^{*} T_{X/G})^\dual, \quad
\ovl{\M}(\Upsilon,X)^* E_{\Gamma'} = (R p_*'' e^{'',*} T_{X/G})^\dual =
(R p_*'' f^* e^* T_{X/G})^\dual .$$
Note that $(x^* e^* T_{X/G})^\dual = \Psi^* L_\Delta$, where $\Delta :
(X/G) \to (X/G)^2$ is the diagonal and $\Psi$ is evaluation at the
node.  We have an exact triangle
$$ \Psi^* L_\Delta [-1] \to \ovl{\M}(\Upsilon,X)^* E_{\Gamma'} \to
E_{\Gamma} \to \Psi^* L_\Delta .$$
This gives rise to a morphism of exact triangles as in Example \ref{cutting}. 
%
%
%
%
\noindent 
By compatibility, the virtual fundamental classes are related by $
[\ovl{\M}^G_{n,\Gamma}(C,X)] = \Delta^!  [\ovl{\M}^{G}_{n+2,\Gamma'}(C,X)]
$.

For collapsing an edge for gauged maps, let $\Upsilon: \Gamma' \to
\Gamma$ be a morphism of rooted graphs given by {\em collapsing an
  edge}.  Associated to $\Upsilon$ are morphisms of Artin
resp. Deligne-Mumford stacks
$$ \ovl{\MM}(\Upsilon): \ovl{\MM}_{n,\Gamma'}(C) \to \ovl{\MM}_{n,\Gamma}(C),
\quad \ovl{\M}(\Upsilon): \ovl{\M}_{n,\Gamma'}(C) \to \ovl{\M}_{n,\Gamma}(C) .$$
The first is a regular local immersion, and so defines a class in the
bivariant Chow group $ [\ovl{\MM}(\Upsilon)] \in
A^\dual(\ovl{\MM}_{n,\Gamma'}(C) \to \ovl{\MM}_{n,\Gamma}(C)) .$ As in
Behrend \cite{be:gw}, the relative obstruction theories for
$\ovl{\M}^G_{n,\Gamma}(C,X), \ovl{\M}^G_{n,\Gamma'}(C,X)$ are related by
pull-back.  Following we have 
\cite[p. 15]{be:gw} 
%
%
$$
\ovl{\MM}(\Upsilon)^!  [\ovl{\M}^G_{n,\Gamma}(C,X,d)] = 
\ovl{\cF}(\Ups,X)_* \sum_{d'
  \mapsto d} [\ovl{\M}^G_{n,\Gamma'}(C,X,d')] $$ 
as claimed.
\end{proof} 

We now pass to homology/cohomology.  (One could also consider the
quantum Chow ring etc.)  Pairing with the virtual fundamental class
gives a map
$$\int_{[\ovl{\M}_n^G(C,X)]}: H(\ovl{\M}_n^G(C,X),\Q) \to \Q .$$
Evaluation at the marked points gives a morphism
$$ \ev: \ovl{\M}_n^G(C,X) \to (X/G)^n, \quad
(P,\hat{C},u,z_1,\ldots,z_n) \mapsto ( z_j^* P, u \circ z_j)_{j=1}^n.$$
Forgetting the bundle and curve and collapsing any unstable components
defines a forgetful morphism from \cite[Corollary 5.19]{qk2}
$ f:
\ovl{\M}_n^G(C,X) \to \ovl{\M}_n(C) .$

\begin{definition}  {\rm (Gauged Gromov-Witten invariants of a given combinatorial type)} 
\begin{enumerate} 
\item {\rm (Invariants for a tree with a single vertex)}
The {\em gauged Gromov-Witten invariants} associated
to $X$ are the maps
$$ H_G(X,\Q)^{n} \times H(\ovl{\M}_n(C),\Q) \to \Q, \quad
  (\alpha, \beta) \mapsto \lan \alpha, \beta \ran_{d} $$
$$ \lan \alpha; \beta \ran_{d} := \int_{[\ovl{\M}^G_{n}(C,X,d)]} \ev^*
\alpha \cup f^* \beta .$$
\item {\rm (Invariants for a connected tree)} The invariant for a {\em
  connected} rooted $H_2^G(X)$-labelled tree $\Gamma$ and
  $G$-equivariant vector bundle $E \to X$ is the integral $\lan
  \alpha, \beta \ran_{E,\Gamma,d} \in \Q$ of $\ev^* \alpha \cup f^*
  \beta \cup \eps(E)$ over the moduli stack
  $\ovl{\M}^G_{n,\Gamma}(C,X)$ of stable gauged maps of combinatorial
  type $\Gamma$.
\item {\rm (Invariants for forests)} Invariants for possibly
  $H_2^G(X)$-labelled rooted forests are defined as follows, given
  the {\em additional datum} of a map from the non-root components to
  the root edges: Suppose that $\Gamma = \Gamma_0 \cup \Gamma_1 \cup
  \dots \cup \Gamma_l$ is a rooted $H_2^G(X)$-labelled forest such
  that each tree $\Gamma_j$ has semi-infinite edges $I_j$, and for
  each $ j = 1,\dots, l$ is given a semi-infinite edge $e(j)$ of
  $\Gamma_0$.  We define gauged Gromov-Witten invariants for $\Gamma$
  by fiber integration over the map $ \ovl{\M}_{n,\Gamma}^G(C,X) \to
  \ovl{\M}_{n_0,\Gamma_0}^G(C,X) $ whose fibers are moduli stack of
  stable maps of type $\Gamma_j$ for $j > 0$: set
$$ \lan \alpha ; \beta \ran_{\Gamma,d} := \lan (\alpha'_i)_{i \in I_0}
  ; \beta \ran_{\Gamma_0,d_0} $$
where for each semi-infinite edge $i$ of $\Gamma_0$
$$\alpha'_i = \left( \prod_{i = e(j)} \lan (\alpha_e)_{e \in I_j} ,
\beta_e \ran_{\Gamma_j,d_j} \right) \alpha_i ,$$
using the $H(BG)$-module structure on $H_G(X)$, and $\beta_j \in
H(\ovl{M}_{n_j,\Gamma_j}(C))$ is the component of $\beta$ in the
decomposition $\ovl{\M}_{n,\Gamma}(C) = \prod_j
\ovl{\M}_{n_j,\Gamma_j}(C)$.
\end{enumerate} 
\end{definition} 

\begin{remark}  It is not possible to define invariants for forests (as opposed
to trees) as purely a product over the tree components, since the
non-root components resp. root component defines invariants with
values in $H(BG) \otimes \Lambda_X^G$ resp. $\Lambda_X^G$.
\end{remark} 

 These invariants (with or without cohomological twisting) satisfy
 axioms for morphisms of rooted trees:
\begin{proposition} \label{alggaugedsplitprop} 
\begin{enumerate}
\item {\rm (Cutting edges)} 
If $\Gamma'$ is obtained from $\Gamma$ by cutting
an edge then 
$$ \lan \alpha ; \beta \ran_{\Gamma,d} =
\sum_{i=1}^{\dim(H(X))} \lan \alpha, \delta_i, \delta^i ; 
\ovl{\M}(\Ups)^* \beta \ran_{\Gamma',d}$$
where $\delta_i,\delta^i$ are dual bases for $H_G(X)$ over $H(BG)$;
\item {\rm (Collapsing edges)} 
If $\Ups: \Gamma \to \Gamma'$ is a morphism collapsing an edge
  then
$$ 
\lan \alpha ; \beta \cup \gamma 
\ran_{\Gamma',d'} =
  \sum_{d' \mapsto d} \lan \alpha ; 
\ovl{\M}(\Ups)^* 
\beta \ran_{\Gamma,d}$$
where $\gamma$ is the dual class for $\ovl{\M}(\Ups):
\ovl{\M}_{n,\Gamma}(C) \to \ovl{\M}_{n,\Gamma'}(C)$.
\item {\rm (Forgetting tails)}  If $\Ups: \Gamma \to \Gamma'$ is a morphism forgetting a tail
 then for $\alpha' \in H^2_G(X), \alpha \in H_G(X)^{n}$,
$$ \lan \alpha, \alpha' ; \ovl{\M}(\Ups)^* \beta \ran_{\Gamma,d} =
  (d, \alpha') \lan \alpha ; \beta \ran_{\Gamma',d'}$$
\end{enumerate} 
\end{proposition} 

\begin{proof} 
By Proposition \ref{gaugedvfcprops} and the same arguments in the
Gromov-Witten case, see Behrend-Manin \cite[Theorem 9.2]{bm:gw}.
\end{proof}

\begin{definition} 
Denote by $\Lambda_X^G$ the {\em equivariant Novikov field} for $X$,
the set of all maps $a: H_2^G(X):=H_2^G(X,\Q) \to \Q$ such that for
every constant $c$, the set of classes
$$ \left\{ d \in H_2^G(X), \ \lan [\omega_{X,G}], d \ran \leq c \right\} $$
on which $a$ is non-vanishing is finite.  Addition is defined in the
usual way and multiplication is convolution.
\end{definition} 

From now on, we denote by $QH_G(X) = H_G(X) \otimes \Lambda_X^G$ the
quantum cohomology over the Novikov field $\Lambda_X^G$.  Summing over
equivariant homology classes gives a map
$$ \tau^G_{X,n}: QH_G(X)^{n} \times H(\ovl{\M}_n(C)) \to \Lambda_X^G ,
\ \ \ (\alpha,\beta) \mapsto \sum_{d \in H_2^G(X,\Z)} q^d \lan \alpha
, \beta \ran_{d} .$$
By Proposition \ref{alggaugedsplitprop}, 

\begin{theorem}  \label{deftrace} 
If $\ovl{\M}_n^{G}(C,X)$ is a Deligne-Mumford
stack (that is, if stable=semistable) then the maps $(\tau^G_{X,n})_{
  n \ge 0}$ form a trace on the CohFT algebra $QH_G(X)$.
\end{theorem} 

Twisted gauged Gromov-Witten invariants are defined as follows.

\begin{definition}  {\rm (Twisting class and twisted gauged invariants)} 
Let $E \to X$ be a $G$-equivariant complex vector bundle, inducing a
vector bundle on $X/G$.  Pull-back under the evaluation map $e:
\ovl{\cC}_n^G(C,X) \to X/G$ gives rise to a vector bundle $e^{*} E \to
\ovl{\cC}_{n}^G(C,X)$, which we can push down to an {\em index} complex
$$\Ind(E) := Rp_{*} e^* E$$ 
in the derived category of bounded complexes of coherent sheaves on
$\ovl{\M}_n^G(C,X)$.  As in \cite[Appendix]{co:qrr}, $\Ind(E)$ admits a
resolution by vector bundles and we may define the
$\C^\times$-equivariant Euler class
$$ \eps(E) := \Eul_{\C^\times}(\Ind(E)) \in
H_G(\ovl{\M}_n^G(C,X)) \otimes \Q[\varphi,\varphi^{-1}]$$
where $\varphi$ is the parameter for the action of $\C^\times$ by scalar
multiplication in the fibers.  The {\em twisted gauged Gromov-Witten
  invariants} associated to $E \to X, C$ are the maps
\begin{multline} H_G(X)^n \times H(\ovl{M}_n(C)) \to 
 \Q[\varphi,\varphi^{-1}], \\ (\alpha,\beta) \mapsto \lan \alpha , \beta \ran_{\Gamma,E,d} =
 \int_{[\ovl{\M}_n^G(C,X,d)]} \ev^* \alpha \cup f^* \beta \cup \eps(E) .\end{multline}
\end{definition} 

\begin{example}   
Recall that in the case that $X$ is a vector space, $G$ is a torus,
$\ovl{\M}^G(C,X)$ is the toric variety $X(d)$, see \cite[(31)]{qk2}.
Integrals over toric varieties may be computed via
residues, as in for example Szenes-Vergne \cite{sv:tr}.  Some sample
computations are computed in Morrison-Plesser \cite[Section 4]{mp:si},
who made contact with the Gelfand-Kapranov-Zelevinsky theory of
hypergeometric functions.  We return to this case in Example
\ref{Ifunction}.
\end{example} 


 \section{Quantum Kirwan morphism and the adiabatic limit theorem} 

In this section we explain how to ``quantize'' the classical Kirwan
morphism in order to obtain a morphism of CohFT algebras to the
quantum cohomology of the quotient.  The existence of such a morphism
was noted under ``sufficiently positive'' conditions on the first
Chern class in Gaio-Salamon \cite{ga:gw}.  The quantum Kirwan morphism
relates small quantum cohomologies under suitable positivity
assumptions.  We also give a partial computation of the quantum Kirwan
map in the toric case.

\subsection{Affine gauged Gromov-Witten invariants} 

We first define gauged affine Gromov-Witten invariants by integrating
pull-back and universal classes over the moduli stack of affine gauged
maps.  As in Behrend \cite{be:gw}, we separate the splitting axiom
into a {\em cutting edges} and {\em collapsing edges} axiom.  The main
difference with Behrend \cite{be:gw} is that one cannot cut an
arbitrary edge and still have a colored tree if the edge separates
some of the colored vertices from the root edge and not others, so
there is a new {\em cutting edges with relations} axiom which cuts
several edges at once.  There is also a difference in the {\em
  collapsing edges} axiom: because the source moduli space
$\ovl{\M}_{n,1}(\bA)$ is not smooth, not every boundary divisor is
Cartier and so there is a new {\em collapsing edges with relations}
axiom which holds for combinations of boundary divisors that are
Cartier.  Let $X$ be a smooth polarized quasiprojective variety such
that the git quotient $X \qu G$ is a (necessarily smooth)
Deligne-Mumford stack.

\begin{definition} {\rm (Virtual fundamental classes for affine gauged maps)} 
\begin{enumerate} 
\item {\rm (Virtual fundamental class for a colored tree with a single
  vertex)} The construction in \cite[Chapter 7]{bf:in} gives a virtual
  fundamental class $[\ovl{\M}_{n,1}^G(\bA,X,d)] \in
  A(\ovl{\M}_{n,1}^G(\bA,X,d))$.  
\item {\rm (Virtual fundamental class for a connected colored tree)}
  More generally, for any combinatorial type of colored tree $\Gamma$
  we have a virtual fundamental class
$$[\ovl{\M}^G_{n,1,\Gamma}( \bA,X,d)] \in A(\ovl{\M}^G_{n,1,\Gamma}(\bA,X,d)) .$$
\item {\rm (Virtual fundamental class for a disconnected colored
  forest)} Suppose that $\Gamma = \Gamma_0 \cup \ldots \cup \Gamma_l$
  with $\Gamma_0$ a possibly disconnected union of components each
  with at least one colored vertex, and $\Gamma_1,\ldots,\Gamma_l$
  connected components with $\Ve(\Gamma_j) \subset \Ve^0(\Gamma)$.
  Suppose that for each component $\Gamma_j$ we are given a non-root
  edge $e(j)$ of $\Gamma_0$.  We denote by
$$ \ovl{\M}^G_{n,\Gamma}(\bA,X,d) := 
\cup_{d = d_0 + \ldots + d_l} 
 \left(  \ovl{\M}^{G,\fr}_{n_0,\Gamma_0}(\bA,X,d_0) \times_{G^{n_0}} \prod_{j=1}^l 
  \ovl{\M}_{0,n,\Gamma_j}(X,d_j) \right) $$
the fiber product determined by the mapping $e$ above.  We have
virtual fundamental classes
$$ [\ovl{\M}^G_{n_0,\Gamma_0}(\bA,X,d_0)] \in
A(\ovl{\M}^G_{n_0,\Gamma_0}(\bA,X,d)) \cong
A_{G^{n_0}}(\ovl{\M}^{G,\fr}_{n_0,\Gamma_0}(\bA,X,d)) $$
given the product of virtual fundamental classes of the components and
equivariant virtual fundamental classes
$$ [\ovl{\M}_{0,n_j,\Gamma_j}(X,d_j)] \in
A_G(\ovl{\M}_{0,n_j,\Gamma_j}(X,d_j)) .$$
These give a virtual fundamental class
$$ [ \ovl{\M}^G_{n,\Gamma}(\bA,X,d)] = 
\cup_{d = d_0 + \ldots + d_l} [
  \ovl{\M}^{G,\fr}_{n_0,\Gamma_0}(\bA,X,d_0)] \times \prod_{j=1}^l [
  \ovl{\M}_{0,n,\Gamma_j}(X,d_j)] $$
in 
$$ A_{G^{n_0}}(\ovl{\M}^{G,\fr}_{n_0,\Gamma_0}(\bA,X,d) \times \prod_j
\ovl{\M}_{0,n_j}(X,d_j)) \cong A(\ovl{\M}^{G}_{n,\Gamma}(\bA,X,d)) .$$
Note that it is not possible to define the virtual fundamental classes
without the additional labelling, since the virtual fundamental
classes for the components $\Gamma_j$ are equivariant while that for
$\Gamma_0$ is not.
\end{enumerate} 
\end{definition}

These classes satisfy the following properties:

\begin{proposition}  \label{affinevfcprops}
\begin{enumerate} 
\item {\rm (Collapsing edges)} If $\Gamma'$ is obtained from $\Gamma$
  by collapsing an edge and $\Ups: \Gamma \to \Gamma'$ is the
  corresponding morphism of colored trees then
$$ \ovl{\MM}(\Ups)^!  [\ovl{\M}^G_{n,1,\Gamma'}(\bA,X,d')] =
  \ovl{\cF}(\Ups,X)_* \sum_{d \mapsto d'}
      [\ovl{\M}^G_{n,1,\Gamma}(\bA,X,d)] $$
where 
$$\ovl{\cF}(\Ups,X): \ovl{\M}^G_{n,1,\Gamma}(\bA,X,d) \to
\ovl{\M}_{n,1,s(\Gamma)}(\bA) \times_{\ovl{\M}_{n,1,s(\Gamma')}(\bA)}
\ovl{\M}^G_{n,1,\Gamma'}(\bA,X,d') $$
is the identification with the fiber product; 
\item {\rm (Collapsing edges with relations)} If
  $\Gamma_0,\ldots,\Gamma_r$ are obtained from $\Gamma$ by collapsing
  edges with relations and $\Ups: \Gamma_0 \sqcup \ldots \sqcup
  \Gamma_r \to \Gamma$ is the corresponding morphism of colored trees
  so that $\cup_{i=1}^r \ovl{\MM}_{n,1,\Gamma_i}^{\tw}(\bA) \to
  \ovl{\MM}_{n,1,\Gamma'}^{\tw}(\bA)$ is a regular local immersion (that is,
  is a Cartier divisor) then
$$ \ovl{\MM}(\Ups)^!  [\ovl{\M}^G_{n,1,\Gamma}(\bA,X,d)] =
  \ovl{\cF}(\Ups,X)_* 
\sum_{d \mapsto d',i=1,\ldots,r}
     [\ovl{\M}^G_{n,1,\Gamma}(\bA,X,d)] .$$

\item {\rm (Cutting edges or edges with relations)} If $\Ups: \Gamma
  \to \Gamma'$ is a morphism of trees of type cutting an edge or
  edges with relations then
$$ \mathcal{G}(\Ups,X)_* [\ovl{\M}^G_{n,\Gamma'}(\bA,X,d')] = \Delta^!
  [\ovl{\M}^{G}_{n,\Gamma}(\bA,X,d)] $$
where $\Delta: \ovl{I}_{X/G}^m \to \ovl{I}_{X/G}^{2m} $ is the diagonal
and $\mathcal{G}(\Ups,X)$ is the gluing morphism in \cite[(32)]{qk2}.
%
\item {\rm (Forgetting tails)} If $\Ups: \Gamma \to \Gamma'$ is a
  morphism forgetting a tail then
$$ \ovl{\M}(\Ups)^! [\ovl{\M}_{n,\Gamma'}^G(\bA,X,d)] =
  [\ovl{\M}^G_{n+1,\Gamma}(\bA,X,d)] .$$
\end{enumerate} 
\end{proposition} 

\begin{proof}   Cutting an edge is similar to the case of gauged
maps from projective curves covered in Proposition
\ref{gaugedvfcprops} and omitted.  For collapsing an edge, Let $\Ups:
\Gamma' \to \Gamma$ be a morphism of edge-rooted colored trees given
by {\em collapsing an edge} 
connecting vertices of the same color.
Associated to $\Ups$ are morphisms of Artin resp. Deligne-Mumford
stacks
$$ \ovl{\MM}(\Ups): \ovl{\MM}_{n,1,\Gamma'}^{\tw}(\bA) \to
\ovl{\MM}_{n,1,\Gamma}^{\tw}(\bA), \quad \ovl{\M}(\Ups):
\ovl{\M}_{n,1,\Gamma'}(\bA) \to \ovl{\M}_{n,1,\Gamma}(\bA) .$$
%
%
 As in Behrend \cite{be:gw}, the relative
 obstruction theories for $\ovl{\M}^G_{n,1,\Gamma}(\bA,X),
 \ovl{\M}^G_{n,1,\Gamma'}(\bA,X)$ are related by pull-back:
$$ \ovl{\MM}(\Ups)^!
[\ovl{\M}^G_{n,1,\Gamma'}(\bA,X,d')] =   \ovl{\cF}(\Ups,X)_* 
\sum_{d
    \mapsto d'} [\ovl{\M}^G_{n,1,\Gamma}(\bA,X,d)] .$$
%
For collapsing several edges, let $\Gamma_0,\ldots,\Gamma_r$ be
colored trees obtained from $\Gamma$ by collapsing edges by morphisms
$\Ups_1,\ldots, \Ups_r$ so that $\cup_{i=1}^r
\ovl{\MM}^{\tw}_{n,1,\Gamma_i}(\bA) \to
\ovl{\MM}_{n,1,\Gamma'}^{\tw}(\bA)$ is a regular local immersion (that
is, is a Cartier divisor).  Then
$$ \ovl{\MM}(\Ups)^!
[\ovl{\M}^G_{n,1,\Gamma'}(\bA,X,d')] = 
  \ovl{\cF}(\Ups,X)_*  \sum_{d
    \mapsto d'} [\ovl{\M}^G_{n,1,\Gamma}(\bA,X,d)] $$
as in Example \ref{collapsing}.  The last item is left to the reader.
\end{proof} 

To define invariants, note that evaluation at the marked points
defines a map
$$ \ev \times \ev_\infty: \ovl{\M}_{n,1}^G(\bA,X) \to (X/G)^n \times
\ovl{I}_{X \qu G} .$$
By integration over the moduli stacks of affine gauged maps we obtain
affine gauged Gromov-Witten invariants defining the quantum Kirwan
morphism of CohFT algebras from $QH_G(X)$ to $QH(X \qu G)$.

\begin{definition} {\rm (Affine gauged Gromov-Witten invariants)} 
\begin{enumerate} 
\item {\rm (Invariants for a connected colored tree)} The {\em affine
  gauged Gromov-Witten invariants} for a connected colored tree
  $\Gamma$ are the maps
\begin{multline} 
 H_G(X)^n \times H(X \qu G) \times H(\ovl{M}_{n,1}(\bA)) \to \Q, \\
 (\alpha,\alpha_\infty,\beta) \mapsto \lan \alpha; \alpha_\infty; \beta
 \ran_{\Gamma,d} := \int_{[\ovl{\M}_{n,1,\Gamma}^G(\bA,X,d)]} \ev^*
 \alpha \cup f^* \beta \cup \ev^*_\infty \alpha_\infty .\end{multline}
\item {\rm (Invariants for a colored forest)} Invariants for possibly
  disconnected $H_2^G(X)$-labelled colored forests are defined as
  follows, given the {\em additional datum} of a map from the non-root
  components to the root edges: Suppose that $\Gamma = \Gamma_0 \cup
  \Gamma_1 \cup \dots \cup \Gamma_l$ is a disconnected colored
  $H_2^G(X)$-labelled tree such that each component of $\Gamma_0$
  has at least one vertex in $\Ve^0(\Gamma)$ or $\Ve^1(\Gamma)$, for
  $j > 1$ the tree $\Gamma_j$ has semi-infinite edges labelled $I_j$,
  and for each $ j = 1,\dots, l$ is given a semi-infinite edge $e(j)$
  of $\Gamma_0$.  Let $ \Edge(\Gamma) = \Edge^0(\Gamma) \cup
  \Edge^\infty(\Gamma) $ denote the partition corresponding to nodes
  mapping to $X/G$ or $I_{X \qu G}$, that is, edges connecting
  $\Ve^0(\Gamma)$ with $\Ve^0(\Gamma) \cup \Ve^1(\Gamma)$ or edges
  connecting $\Ve^1(\Gamma) \cup \Ve^\infty(\Gamma)$ with
  $\Ve^\infty(\Gamma)$ as in 
\cite[Remark 2.25]{qk1}. 
 We suppose that we have a labelling of the semi-infinite edges by
 classes $ \alpha_e \in H_G(X), e \in \Edge^0(\Gamma)$ and $ \alpha_e
 \in H(I_{X \qu G}), e \in \Edge^\infty(\Gamma)$.  We define gauged
 Gromov-Witten invariants for $\Gamma$ by fiber integration over the
 map $ \ovl{\M}_{n,\Gamma}^G(\bA,X) \to \ovl{\M}_{n_0,\Gamma_0}^G(\bA,X)
 $ whose fibers are moduli stacks of stable maps of type $\Gamma_j$
 for $j > 0$: set
$$ \lan \alpha ; \beta \ran_{\Gamma,d} := \lan (\alpha'_j)_{j \in I_0}
  ; \beta_0 \ran_{\Gamma_0,d_0} $$
where for each semi-infinite edge $i$ of $\Gamma_0$ connecting to a
vertex in $\Ve^0(\Gamma_0)$ or $\Ve^1(\Gamma_0)$,
$$\alpha'_i = \left( \prod_{i = e(j)} \lan (\alpha_e)_{e \in I_j} ,
\beta_j \ran_{\Gamma_j,d_j} \right) \alpha_i ,$$
using the $H(BG)$-module structure on $H_G(X)$, and $\beta_j$ is the
K\"unneth component of $\beta$ for the component $\Gamma_j$.
\item {\rm (Twisted affine Gromov-Witten invariants)} Twisted
  invariants $\lan \alpha ; \beta \ran_{\Gamma,d,E}$ associated to
  $G$-equivariant vector bundles $E \to X$ are defined by inserting
  Euler classes of indices $\eps(E)$ into the integrands.
\end{enumerate} 
\end{definition}

The properties of the affine Gromov-Witten invariants are similar to
those for the projective case:

\begin{proposition} \label{affinegwprops}
\begin{enumerate} 
\item {\rm (Collapsing an edge)} If $\Gamma'$ is obtained from
  $\Gamma$ by collapsing an edge then for any labelling $d'$ of
  $\Gamma'$, and $\ovl{\M}(\Ups): \ovl{\M}^{\tw}_{n,1,\Gamma}(\bA) \to
  \ovl{\M}^{\tw}_{n,1,\Gamma'}(\bA)$ has dual class $\gamma$ then
$$ \langle \alpha ; \beta \cup \gamma \rangle_{\Gamma',d',E} = \sum_{d
    \mapsto d'} \langle \alpha ; \ovl{\M}(\Ups)^* \beta
  \rangle_{\Gamma,d,E} $$
\item {\rm (Collapsing edges with relations)} More generally, if
  $\Gamma_0,\ldots,\Gamma_r$ are each obtained from $\Gamma$ by
  collapsing edges with relations and $\Ups: \Gamma_0 \sqcup \ldots
  \sqcup \Gamma_r \to \Gamma$ is the corresponding morphism of colored
  trees so that
\begin{equation} \label{immerse}
\cup_{i=1}^r \ovl{\M}_{n,1,\Gamma_i}^{\tw}(\bA) \to
  \ovl{\M}_{n,1,\Gamma'}^{\tw}(\bA)\end{equation} 
is a regular local immersion (that is, is a Cartier divisor) with dual
class $\gamma$ then
$$ \langle \alpha ; \beta \cup \gamma \rangle_{\Gamma',d',E} = \sum_{d
    \mapsto d',i=1,\ldots,r} \langle \alpha ;
  \iota_{\Gamma_i,\Gamma}^* \beta \rangle_{\Gamma_i,d,E} $$
where $\iota_{\Gamma_i,\Gamma}^*$ are the components of
\eqref{immerse}.
\item {\rm (Cutting an edge)} If $\Gamma'$ is obtained from $\Gamma$
  by cutting an edge or edges with relations then
$$ \langle \alpha ; \beta \rangle_{\Gamma,d,E} = \sum_k \langle
  \alpha, \delta_k \cup \Eul_{G \times \C^\times}(E), \delta^k; \ovl{\M}(\Ups)^* \beta \rangle_{\Gamma',d,E} $$
where $(\delta_k) ,(\delta^k), {k=1, \ldots, \dim(H(I^m_{X\qu G}))}$
are dual bases for $H(I_{X \qu G}^m)$ resp.  $H_G(X)$ if the cut edges
lie in $\Edge^\infty(\Gamma)$ resp.  $\Edge^0(\Gamma)$.
\item {\rm (Forgetting tails)} If $\Ups: \Gamma \to \Gamma'$ is a
  morphism forgetting a tail then
$$ \lan \alpha, \alpha' ; \ovl{\M}(\Ups)^* \beta \ran_{\Gamma,d,E} =
  (d, \alpha') \lan \alpha ; \beta \ran_{\Gamma',d,E}$$
where $(d,\alpha')$ is the pairing between $d \in H^G_2(X,\Q)$ and
$\alpha' \in H^2_G(X,\Q)$.
\end{enumerate} 
\end{proposition} 
  
\begin{proof}  By Proposition \ref{affinevfcprops}; the cutting edges case follows from an integration
over the fiber $\ovl{\M}_{n,\Gamma}^G(\bA,X) \to
\ovl{\M}_{n,\Gamma_0}^G(\bA,X)$ with fibers $\prod_{j > 0}
\ovl{\M}_{0,n,\Gamma_j}(X)$.  The collapsing edges and forgetting
tails properties are left to the reader.
\end{proof}

\subsection{Quantum Kirwan morphism} 

In this section we use the affine gauged Gromov-Witten invariants to
define the quantum Kirwan morphism from $QH_G(X)$ to $QH(X \qu G)$.
For simplicity, we restrict to the case $E$ trivial, that is, the
untwisted case.  We remind that here $QH(X \qu G)$ is defined over the
equivariant Novikov ring, that is, $QH(X \qu G) = H(X \qu G) \otimes
\Lambda_X^G$.

\begin{definition} {\rm (Quantum Kirwan morphism)}   
Suppose that $X$ is a smooth polarized projective $G$-variety or a
vector space with a linear action of $G$ and proper moment map such
that the git quotient $X \qu G$ is a Deligne-Mumford stack, so that
the moduli stacks $\ovl{\M}^G_n(\bA,X)$ are proper Deligne-Mumford
stacks. 
 The {\em quantum Kirwan morphism} is the collection of maps
$$ \kappa_X^{G,n}: QH_G(X)^{n} \times H(\ovl{\M}_{n,1}(\bA)) \to QH(X
\qu G) , n \ge 0 $$
given by pull-back to $\ovl{\M}^G_{n,1}(\bA,X)$ and push-forward to $X
\qu G$.  That is, for $\alpha \in H_G(X)^{n}, \alpha_\infty \in
H_G(\ovl{I}_{X \qu G}), \beta \in H^*(\ovl{\M}_{n,1}(\bA))$ let
$$ ({\kappa}_X^{G,n} (\alpha,\beta), \alpha_\infty) = \sum_{d \in
  H_2^G(X,\Q)} q^d  \lan \alpha; \alpha_\infty; \beta \ran_d $$
using Poincar\'e duality; the pairing on the left is given by cup
product and integration over $\ovl{I}_{X \qu G}$.  Define ${\kappa}_G^0
\in H^*(\ovl{\M}_{n,1}(\bA))$ similarly, by integrating the unit.
\end{definition}

\begin{theorem}  The collection $\kappa_X^G = ({\kappa}_X^{G,n})_{n \ge 0}$ 
satisfies the axioms of a morphism of CohFT algebras.
\end{theorem} 

\begin{proof}  First note that the splitting axiom is well-defined: 
 Note that $\kappa_X^{G,0}$ has contributions with coefficients $q^d$
 with $(d,[\omega_{X,G}]) > 0$, since trivial maps with no finite
 markings are unstable.  It follows that the sum on the
 right-hand-side of \cite[(11)]{qk1}
 is finite modulo terms with coefficient $q^a$ and higher, for any $a
 \in \R$.  The equation \cite[(11)]{qk1}
now follows from parts (a)-(c) of Proposition
\ref{affinegwprops}.
\end{proof} 

\begin{remark} \label{eqqkirwan}
\begin{enumerate} 
\item {\rm (Equivariant quantum Kirwan morphism)} 
If the action of $G$ extends to an action of a group
$\ti{G}$ containing $G$ as a normal subgroup, there is a map
$$ QH_{\ti{G}}(X)^n \times H(\ovl{\M}_{n,1}^G(\bA,X))
\to QH_{\ti{G}/G}(X \qu G)
$$
defined by the same formula.  After extending the coefficient ring of
$QH_{\ti{G}/G}(X \qu G)$ from $\Lambda_{X \qu G}$ to $\Lambda_X^{G}$
we have a morphism of CohFT algebras
\begin{equation} \label{qpartial}
({\kappa}_X^{\ti{G},G,n})_{n \ge 0} : QH_{\ti{G}}(X) \to
    QH_{\ti{G}/G}(X \qu G) .\end{equation}
\item {\rm (Flatness of the quantum Kirwan morphism in the positive
  case)} Suppose that $c_1^G(X)$ is semipositive in the sense that
  $(c_1^G(X), d) \ge 0$ for the homology class $d$ of any gauged
  affine map.  In this case, the ``quantum corrections'' in any
  $\kappa_X^{G,n}(\alpha_1,\ldots,\alpha_n)$ are of degree at most
  $\deg(\alpha_1) + \ldots + \deg(\alpha_n) + 2 - 2n$.  In particular,
  the element $ \kappa_X^{G,0}(1)$ can be written as the sum of
  elements of degree $0$ and $2$ with respect to the grading induced
  by the grading on $H(I_{X \qu G})$.  If $c_1^G(X)$ is positive, then
  the dimension count shows that $\kappa_X^{G,0}$ is an element of
  degree $0$ in $H(I_{X \qu G})$, times an element of $\Lambda_X^G$,
  that is, a multiple of the point class.  If $(c_1^G(X),d)$ is at
  least two whenever $(d,[\omega_{X,G}]) > 0$ then $\kappa_X^{G,0}$
  vanishes.
  \end{enumerate} 
\end{remark}  

We end this section with a partial computation of the quantum Kirwan
morphism in the toric case.  Suppose that $X \cong \C^k$ is a vector
space equipped with a linear action of a torus $G$ with Lie algebra
$\g$ and weights $\mu_1,\ldots,\mu_k \in \g^\dual$ in the sense that
$G$ acts on the $j$-th factor by the character $\exp(\mu_j)$.  We
denote by $\ti{G} = (\C^\times)^k$ the torus acting on $X$ by scalar
multiplication on each factor.  Let $v_1,\ldots, v_k$ be the standard
coordinates on the Lie algebra $\ti{\g}$ so that
$$ QH_{\ti{G}}(X) = \Q[v_1,\ldots,v_k] \otimes \Lambda_X^{\ti{G}} .$$  
However, for the purposes of this section it suffices to tensor with
the $G$-equivariant Novikov field $\Lambda_X^G$.  The inclusion $G \to
\ti{G}$ induces a map $ r: QH_{\ti{G}}(X) \to QH_G(X)$, which after
identification of the equivariant cohomology with symmetric functions
$QH_G(X) \cong \Sym(\g^\dual) \otimes \Lambda_X^G $ is the restriction
map induced by the inclusion $\g \to \ti{\g}$.  Let $l(v_j), j =
1,\ldots, k$ denote the divisor classes in $H(I_{X \qu G})$ defined by
$v_j$, see \cite[Example 4.8]{qk2}.

\begin{lemma} \label{leadingorder} Let $G$ be a torus acting on a vector space $X$ as above.  For any $d \in H_2^G(X,\Z)$ such that the polarization vector  $\nu$ lies in $ \on{span}
\{ - \mu_j, \mu_j(d) \ge 0 \}  $ (see \cite[(30)]{qk2})
%
we have 
$$\kappa_X^{G,1} \left( \prod_{\mu_j(d) \ge 0} r(v_j)^{\mu_j(d)} \right) =
q^d \prod_{\mu_j(d) \leq 0} l(v_j)^{- \mu_j(d)} + {\rm \ higher\ order} $$
where higher order means terms with coefficient $q^{d'}$ with
$(d',[\omega_{X,G}]) > (d, [\omega_{X,G}])$.
\end{lemma} 

\begin{proof}  We show 
\begin{equation} \label{qsreq}
\int_{[\ovl{\M}_{1,1}^G(\bA,X,d)]} \prod_{ \mu_j(d) \ge 0} r(v_j)^{\mu_j(d)} \cup
\ev_\infty^* \alpha = \int_{[X \qu G]} \prod_{\mu_j(d) \leq 0} l(v_j)^{-\mu_j(d)} \cup \alpha.
 \end{equation}
We compute the left-hand-side by interpreting the first factor as an
Euler class
$$ \prod_{\mu_j(d) \ge 0} r(v_j)^{\mu_j(d)} = \ev^* \Eul \left(\bigoplus_{\mu_j(d) \ge 0}
\C^{\mu_j(d)} \right) $$
and counting the zeros of a section.  Identifying framed maps with a
single marking with maps $u: \bA \to X$, consider the map
$$ \sigma: \ovl{\M}_{1,1}^{G}(\bA,X,d) \to \ev_1^* \prod_{\mu_j(d)
  \ge 0} \C_{\mu_j}^{\mu_j(d)}, \quad u \mapsto
( u_i^{(j)}(0) )_{i=1,j = 1}^{k,  \mu_j(d)  - 1} $$
whose components are the derivatives of the map at the finite marking.
On the stratum $\M_{1,1}^G(\bA,X,d)$ of curves with irreducible
domain, the intersection $ \sigma^{-1}(0)$ maps injectively into $X
\qu G \subset \ovl{I}_{X \qu G}$ via $\ev_\infty$.  Indeed the
assumption on the span of $\mu_j, \mu_j(d) \ge 0$ implies that
$\M_{1,1}^G(\bA,X,d)$ is non-empty, the equation $\ev_\infty(u) =
\pt_{X \qu G}$ fixes the leading order terms (see Examples 
\cite[5.32]{qk2}
and \ref{jtgen3}) and $\sigma(u) = 0$ fixes the lower order terms in
$u$.  Since $\ev_\infty$ maps smoothly onto $\bigoplus_{\mu_j(d) \ge
  0} \C_{\mu_j} \cap X^{\sss} / G$, the integral \eqref{qsreq} is equal
to
$$ \int_{[X \qu G]} \alpha \cup \prod_{\mu_j(d) < 0}
l(v_j)^{-\mu_j(d)} $$
where the virtual integration of $[u]$ is with respect to the virtual
fundamental class induced from that on the moduli stack.  Taking into
account the {\em obstruction bundle} 
$$ R^1 p_* e^* T(X/G) = \ev_1^* \bigoplus_{ \mu_j(d) < 0 }
\C_{\mu_j}^{ - \mu_j(d) - 1} $$
we see that $\kappa_G^{X,1}( \prod r(v_j)^{\max(0,\mu_j(d))}) $
contains a term of the form $q^d \prod l(v_j)^{\max(0,-\mu_j(d))}$
plus contributions from other strata and components of the moduli
space of other homology classes.

We check next that there are no contributions from boundary strata.
On the boundary with curves of reducible domain, each map $u$ consists
of a component $u_1: C_1 \to X/G$ consisting of an affine scaled map
of homology class $d'$ with $(d', [\omega_{X,G}]) < (d,
[\omega_{X,G}])$ connecting the marking $z_1$ to the infinite marking
$z_0$, together with bubbles in $X \qu G$ and possibly other affine
scaled maps.  The vanishing $\sigma(u) = 0$ implies that, in
particular, the $\mu_j(d')$-th derivative of $u_1$ is zero if
$\mu_j(d')$ is integral and less than some non-negative $\mu_j(d)$.
The same conclusion holds if $\mu_j(d') < \mu_j(d)$ is negative, since
in this case the $j$-th component of $u_1$ vanishes identically.  On
the other hand, since $( d- d' , [\omega_{X,G}]) > 0$, the set of
points in $X \qu G$ whose $j$-th coordinate vanishes if $\mu_j(d - d')
> 0$, is unstable, see \cite[(30)]{qk2}.
 Thus, $\sigma^{-1}(0)$ is
empty on the boundary strata and the only contribution to the integral
above arises from the component of maps with irreducible domain.
\end{proof} 

\begin{example}\label{jtgen3} 
\begin{enumerate}
\item {\rm (Projective Space Quotient)} If $G = \C^\times$ acts on $X
  = \C^k$ with all weights one, so that $X \qu G = \P^{k-1}$, then $
  \M^G_{1,1}(\bA,X)$ may be identified with the space of $k$-tuples of
  polynomials $(p_1(z),\ldots,p_n(z))$ with $(p_1(z),\ldots,p_k(z))$
  non-zero for $z$ generic.  We obtain a section
$$ \sigma: \ovl{\M}^G_{1,1}(\bA,X,1) \to \ev_1^* (X \times X \to X) $$ 
by evaluating the polynomials at $0$.  This section has no zeroes
other than at $ [ c_1 z, \ldots, c_k z]$ for $(c_1,\ldots, c_k) \neq
0$, which lies in the open stratum of maps with irreducible domain.
In particular,
$$   \int_{[ \ovl{\M}^G_{1,1}(\bA,X,1)]} \ev_1^* (X \times X \to X)
\cup \ev_\infty^* ([ \pt_{X \qu G}]) = 1 $$
which implies that $\kappa_X^{G,1}(\xi^k) = q$ where $\xi$ is the
generator of $QH_G(X) \cong \Lambda_X^G[\xi]$.
\item {\rm (Weighted Projective Line Quotient)} Let $X = \C_2 \oplus
  \C_3$ and $G = \C^\times$ so that $X \qu G = \P[2,3]$.  Let
  $\theta_1$ resp. $\theta_2$ resp. $\theta_3$ resp. $\theta_3^2$
  denote the generator of the component of $QH(X \qu G) \cong
  H(\ovl{I}_{X \qu G}) \otimes \Lambda_{X \qu G}$ with trivial isotropy
  resp. $\Z_2$ isotropy resp. corresponding to $\exp( \pm 2 \pi i /3)
  \in \Z_3$.  Let $\xi \in H^2_G(X)$ denote the integral generator.
  One has 
$$
  \kappa_X^{G,1}(1) = 1, \quad \kappa_X^{G,1}(\xi) = \theta_1, \quad
  \kappa_X^{G,1}(\xi^2) = q^{1/3} \theta_3/6, $$ $$ \kappa_X^{G,1}(\xi^3) = q^{1/2} \theta_2/18, 
\quad \kappa_X^{G,1}(\xi^4) =
q^{2/3} \theta_3^2/36, \quad \kappa_X^{G,1}(\xi^5) = q/108 .$$ 
In particular, we see that $\kappa_X^{G,1}$ is surjective and the
kernel is $\xi^5 - q/108$, hence 
$$ QH(\P[2,3]) = \Q[\xi] \otimes \Lambda_X^G / (\xi^5 - q/108) $$
which is a special case of Coates-Lee-Corti-Tseng \cite{coates:wps}.
\end{enumerate} 
\end{example}

\begin{remark}  
\begin{enumerate} \label{stages} 
\item {\rm (Quantum Kirwan surjectivity)} We conjecture the quantum
  analog of Kirwan surjectivity, namely that $\kappa_X^{G,1}$ is
  surjective onto the orbifold quantum cohomology $QH(X \qu G)$ of the
  quotient $X \qu G$.  We have worked out some special cases with
  Gonz\'alez in \cite{gw:surject}.
\item {\rm (Quantum reduction in stages)} One naturally expects a
  quantum analog of the reduction in stages theorem: If $G' \subset G$
  is a normal subgroup then $\kappa_{X \qu G'}^{G/G'} \circ
  \kappa_X^{G,G'} = \kappa_X^{G}: QH_G(X) \to QH(X \qu G) $.  That is,
  we have a commutative diagram of CohFT algebras
$$\begin{diagram} \node{QH_G(X)} \arrow{se,b}{\kappa_X^{G,G'}} \arrow[2]{e,l}{\kappa_X^{G}} \node{}
\node{QH(X \qu G)} \\ \node{} \node{QH_{G/G'}(X \qu G')}
\arrow{ne,r}{\kappa_{X \qu G'}^{G/G'}}
\node{} \end{diagram}.$$
\end{enumerate} 
\end{remark}  

There is a $\C^\times$-equivariant extension of the quantum Kirwan map
$$ \kappa_X^G: QH_G(X) \to QH_{\C^\times}(X \qu G) $$
where $QH_{\C^\times}(X \qu G)$ denotes the completed
$\C^\times$-equivariant cohomology, isomorphic to $QH(X \qu G)[[\zeta
]] $ where $\zeta$ is the equivariant parameter.  The construction is
by pushing-forward over $\ovl{\M}_n^G(\bA,X)$
$\C^\times$-equivariantly, as follows.  Choose a base point in $\bA$,
inducing an identification $\bA \to \C$ and so a $\C^\times$-action on
$\bA$.  The action induces an action on ${\M}_n^G(\bA,X)$, given by
pre-composing each morphism with the action.  Equivalently, the action
is given by acting by scalar multiplication on the one-form $\lambda$,
and so extends to the compactification $\ovl{\M}_n^G(\bA,X)$.  As a
result, the virtual fundamental class for $\ovl{\M}_n^G(\bA,X)$ has a
$\C^\times$-equivariant extension which gives rise to a
$\C^\times$-equivariant extension of $\kappa_X^G$.

\subsection{The adiabatic limit theorem}

We show the adiabatic limit \cite[Theorem 1.5]{qk1}, using a divisor
class relation relating curves with finite and infinite scaling.  Note
that divisor class relations in one-dimensional source moduli spaces
have already been used to prove the associativity of the quantum
products, as well as the homomorphism property of the quantum Kirwan
morphism.

Recall from \cite[Theorem 5.35]{qk2} the stack $\ovl{\M}_{n,1}^G(C,X)$
of scaled gauged maps from $C$ to $X$.  Under the stable=semistable
assumption it has a perfect relative obstruction theorem over
$\ovl{\MM}_{n,1}^{\tw}(C)$, whose complex is dual to $Rp_* e^*
T(X/G)$, and so a virtual fundamental class.

\begin{definition} 
If every polystable gauged map is stable then the {\em scaled gauged
  Gromov-Witten invariants} for $\alpha \in H_G(X)^n, \beta \in
H(\ovl{M}_{n,1}(C))$ are
\begin{equation} \label{scaledinv}
 \lan \alpha, \beta \ran_{d,1,E} = \int_{[\ovl{\M}^G_{n,1}(C,X,d)]}
 \ev^* \alpha \cup f^* \beta \cup \eps(E) .\end{equation}
Define 
$$ \phi^n : QH_G(X)^n \times H(\ovl{M}_{n,1}(C)) \to \Lambda_X^G, \quad
(\alpha,\beta) \mapsto \sum_{d \in H_2^G(X,\Q)} q^d \lan \alpha, \beta
\ran_{d,1,E} $$
for $\alpha \in H_G(X)^n$, extended to $QH_G(X)^n$ by linearity.
\end{definition} 

More generally there are invariants for arbitrary combinatorial type
that satisfy the splitting axioms as in \ref{gaugedvfcprops}.  The
adiabatic limit theorem \cite[1.5]{qk1} follows from the divisor class
relation
\begin{equation} \label{divclass} [\ovl{\M}_n(C)] =
  [\cup_{r,[I_1,\ldots, I_r]} D_{I_1,\ldots, I_r}] \in
  H(\ovl{\M}_{n,1}(C))
 \end{equation}
 from \cite[Proposition 2.43]{qk1}.  Indeed\footnote{This explanation
   was added after publication.} the locus of gauged maps
 of $D_{I_1,\ldots,I_r}$ admits as in \cite[(33)]{qk2} a fibration to
 $\ovl{\M}_{r,1}^G(C,X)$ with fiber a subset of $ \Pi_{j=1}^r
 \ovl{\M}_{i_j,1}^G(\bA,X)$.  To describe this fibration, let
 $\ovl{\M}_{r,1}^{G,\fr}(C,X)$ denote the stack of gauged maps equipped
 with trivializations $T_{z_i}(\hat{C}) \to \C$ of the tangent lines
 at the framings.  Then $D_{I_1,\ldots,I_r}$ is obtained as the
 $(\C^\times)^r$-quotient of the diagonal action
$$ D_{I_1,\ldots,I_r} \cong 
(\ovl{\M}_{r,1}^{G,\fr}(C,X) \times \Pi_{j=1}^r
\ovl{\M}_{i_j,1}^G(\bA,X))/ (\C^\times)^r
 .$$
We have a commutative diagram 
$$ \begin{diagram}
\node{
 (\ovl{\M}_{r,1}^{G,\fr}(C,X)
 \times_{(I_{X \qu G})^r}
 \Pi_{j=1}^r
\ovl{\M}_{i_j,1}^G(\bA,X))} 
\arrow{e} \arrow{s}
\node{ \ovl{\M}_{r,1}^{G,\fr}(C,X)} \arrow{s}  \\
\node{D_{I_1,\ldots, I_r}} \arrow{e}  \node{ \ovl{\M}_{r,1}^{G}(C,X)}
\end{diagram} .$$
The integral of $\ev^*(\alpha)$ over $D_{I_1,\ldots,I_r}$ can be
computed by first pulling back to $ \ovl{\M}_{r,1}^{G,\fr}(C,X)
\times_{(I_{X \qu G})^r} \Pi_{j=1}^r \ovl{\M}_{i_j,1}^G(\bA,X)$,
pushing forward to $\ovl{\M}_{r,1}^{G,\fr}(C,X)$ and taking the
quotient.  The equivariant push-forward produces the cohomology class
$$ \kappa_X^{G,i_1}(\alpha_{I_1}) \otimes \ldots \otimes
\kappa_X^{G,i_r}(\alpha_{I_r}) $$
where $\alpha_{I_k} := \otimes_{i \in I_k} \alpha_i $.  Taking the
quotient identifies the equivariant parameter $\zeta$ with the Chern
class of the tangent line, that is $\zeta = -\psi$.  Thus the
contribution from $D_{I_1,\ldots, I_r}$ is
$$ \tau_{X \qu G}^r( \kappa_X^{G,i_1}(\alpha_{I_1}) |_{\zeta = -\psi}, \ldots,
\kappa_X^{G,i_r}(\alpha_{I_r}) |_{\zeta = -\psi}) .$$
The divisor class formula gives the relation
$$ \tau_X^{G,n}(\alpha_1,\ldots,\alpha_n) = 
\sum_{I_1,\ldots,I_r} \tau_{X \qu G}^r( \kappa_X^{G,i_1}(\alpha_{I_1})
|_{\zeta = -\psi}, \ldots, \kappa_X^{G,i_r}(\alpha_{I_r}) |_{\zeta = -
  \psi}) $$
which is the precise form of the adiabatic limit theorem.  More
generally, the divisor class relations from $\ovl{M}_{n,1}(C)$ pull
back to relations in $\ovl{\M}_{n,1}^G(C,X)$ giving a $2$-morphism
between $\tau_{X \qu G} \circ \kappa_X^G$ and $\tau_X^G$.

\section{Localized graph potentials}

In this section we make contact with the hypergeometric functions
appearing in the work of Givental \cite{gi:eq}, Lian-Liu-Yau
\cite{lly:mp1}, Iritani \cite{iri:gmt} and others.  These results
compute a fundamental solution of the quantum differential equation of
the quotient by studying the contributions to the localization formula
for the circle action on the moduli spaces of gauged maps on the
projective line.  Note that in contrast to \cite{gi:eq},
\cite{lly:mp1} etc., the target can be an arbitrary projective (or in
some cases, quasiprojective) $G$-variety.  The virtual localization
formula expresses the result as a sum over fixed point contributions,
and comparing the contributions to the adiabatic limit \cite[Theorem
  1.5]{qk1} one obtains a stronger result which is closely related to
the ``mirror theorems'' of \cite{gi:eq}, \cite{lly:mp1},
\cite{iri:gmt}, \cite{coates:mirrorstacks},
\cite{ciocan:genuszerowall}.

\subsection{Liouville insertions} 

First we introduce a ``Liouville class'' in the definition of the
graph potential.  This is mostly for historical reasons, to compare
with the results of Givental \cite{gi:eq}.  We first consider the
case of ordinary Gromov-Witten theory with target $X$.  Denote the
universal curve and evaluation map
$$ \begin{diagram} 
 \node{\ovl{\cC}_n(C,X)} \arrow{e,t}{e \times e_C} \arrow{s,l}{p}
\node{ X
\times C} \\
\node{ \ovl{\M}_n(C,X)} \end{diagram} .$$
Let $ [\omega_C] \in H^2(C)$ denote a generator.

\begin{definition}  {\rm (Liouville class and invariants with Liouville insertions) }
 Let $\gamma \in H^2(X)$.  The {\em Liouville class} associated to
 $\gamma$ is
$$ \lambda(\gamma) := \exp( p_* ( e^* \gamma \cup e_C^* [\omega_C] ))
 \in H(\ovl{M}_n(C,X)) .$$
The graph invariants with Liouville insertions are maps
\begin{multline} H(X)^{n} \times H(\ovl{M}_n(C)) 
\otimes H^2(X) \to \Q[\varphi,\varphi^{-1}], \\ \lan \alpha , \beta ,
\gamma \ran_{E,d} = \int_{[\ovl{\M}_n(C,X,d)]} \ev^* \alpha \cup f^*
\beta \cup \lambda(\gamma) \cup \eps(E)\end{multline}
where $\varphi$ is the equivariant parameter for scalar
multiplication.

Similarly for gauged Gromov-Witten invariants, any class $\gamma \in
H_G^2(X)$ gives rise to a {\em gauged Liouville class}
$$ \lambda(\gamma) = \exp(p_* ( e^* \gamma \cup e_C^* [\omega_{C}]))
\in H(\ovl{M}^G_n(C,X)) .$$
Define invariants

\begin{multline} H(X)^{n} \times H(\ovl{M}_n(C)) 
\otimes H^2(X) \to \Q[\varphi,\varphi^{-1}], \\ \lan \alpha , \beta ,
\gamma \ran_{E,d} = \int_{[\ovl{\M}^G_n(C,X,d)]} \ev^* \alpha \cup f^*
\beta \cup \lambda(\gamma) \cup \eps(E).\end{multline}
\end{definition}

\subsection{Localized equivariant graph potentials} 
\label{loc1}

In this section we discuss the extraction of a fundamental solution to
the quantum differential equation from the graph potential, following
e.g. Givental \cite{gi:eq}.  Let $X$ be a smooth projective variety,
or more generally, a smooth proper Deligne-Mumford stack with
projective coarse moduli space.  Let $C = \P$ be equipped with the
standard $\C^\times$ action with fixed points $0,\infty \in \P$.
Denote by $\zeta$ the equivariant parameter corresponding to the
$\C^\times$-action.  The graph potential $\tau_X$ has a natural
$\C^\times$-equivariant generalization
$$ \tau_X^{\C^\times}: QH(X) \times H_{\C^\times}(\ovl{M}_n(\P)) \times H^2(X)
\to \Lambda_X[[\zeta]] .$$
For simplicity we restrict to the untwisted case, that is, $E$
trivial.  The class $[\omega_\P] \in H^2(\P)$ with integral one has a
unique equivariant extension $[\omega_{\P,\C^\times}] \in
H^2_{\C^\times}(\P)$ taking values $0$ resp. $\zeta$ in
$H_{\C^\times}(\on{pt}) \cong \Q[\zeta]$ after restriction to $0$
resp. $\infty $ in $\P$.  The following is well-known, see for example
Givental \cite{gi:eq}.

\begin{proposition}   \label{fpdes}
{\rm (Fixed points for the $\C^\times$-action on graph spaces)} The
induced action of $\C^\times$ on $\ovl{\M}_n(\P,X,d)$ has fixed points
given by configurations $[u]$ consisting of a principal component $C_0
\cong \P$ on which the map $u$ is constant, an $(n_-+1)$-marked stable
map $u_- : C_- \to X$ of degree $d_-$ attached to $0 \in \P$, and an
$(n_++1)$-marked stable map $u_+ : C_+ \to X$ of degree $d_+$ attached
to $\infty \in \P$ with $d_- + d_+ = d$ and $n_- + n_+ = n$.
\end{proposition}

We denote by $F_n(d_-,d_+)$ the locus of the fixed point set with
stable maps of classes $d_-,d_+$ respectively attached at $0,\infty
\in \P$.  It has a canonical map $ F_n(d_-,d_+) \to X $ given by
evaluation at any point on the component where the map is constant.
We denote by $\ovl{\gamma}$ the pull-back of $\gamma$ to
$F_n(d_-,d_+)$.

\begin{lemma} \label{mommap} {\rm (Restriction of Liouville class to 
fixed points) } The restriction of $\lambda(\gamma)$ to $F_n(d_-,d_+)$
  is equal to $\exp(\ovl{\gamma} + (d_+,\gamma) \zeta)$.
\end{lemma}  

\begin{proof}  
The restriction of $[\omega_{\P,\C^\times}]$ to the fixed point $0$
resp. $\infty$ in $\P$ is $0$ resp. $\zeta$.  Hence the restriction of
$ e^* \gamma \cup e_C^* [\omega_{\P,\C^\times}] $ to a fixed map as in
Proposition \ref{fpdes} is given by 
$$ e^* \gamma \cup e_C^* [\omega_{\P,\C^\times}] |_{C_+} =  \zeta e^* \gamma $$ 
for the components $C_+$ attached to $\infty$, and 
$$  e^* \gamma \cup e_C^* [\omega_{\P,\C^\times}] |_{C_-} = 0 $$
for the components $C_-$ attached to $0$.  The push-forward $p_* ( e^*
\gamma \cup e_C^* [\omega_{\P,\C^\times}] )$ is given by integration
over the union $C_- \cup C_0 \cup C_+$, and the integrals may be
computed separately over each component.  There are two components on
which the integrand is non-zero: over the components $C_+$ attached at
$\infty$ the integral is $(\gamma,d_+) \zeta$, while the integral over
the constant component is $\ovl{\gamma}$, since the integral of
$[\omega_{\P,\C^\times}]$ over $\P$ is $1$ by definition.  Hence 
$$p_* ( e^*
\gamma \cup e_C^* [\omega_{\P,\C^\times}] ) = (\gamma,d_+) \zeta + \ovl{\gamma} $$
and the Liouville class is $\exp(\ovl{\gamma} + (d_+, \gamma) \zeta)$.
\end{proof} 

\begin{definition} {\rm (Localized graph potentials)} 
 Define the {\em localized graph potentials} (also known as the
 one-point descendent potential) 
$$ \tau_{X,\pm}: QH(X) \to QH(X)[[\zeta^{-1}]] $$
by push-pull over the fixed point component given by $\ovl{\M}_{0,n +
  1}(X,d)$
$$ \tau_{X,\pm}(\alpha,q,\zeta) := \sum_{n \ge 0} (1/n!)
\tau_{X,\pm}^n(\alpha,\ldots, \alpha,q,\zeta) $$
where for $n \neq 1$
$$ \tau_{X,\pm}^n(\alpha_1,\ldots,\alpha_n,q,\zeta) = \sum_{d \in
  H_2^G(X,\Z)} q^d \ev_{n+1,*} \left( \mp (\zeta (\mp \zeta -
\psi_{n+1}))^{-1} \bigcup_{i=1}^n \ev_i^* \alpha_i \right),$$
$\psi_{n+1} \in H^2(\ovl{\M}_{0,n+1}(X,d))$ is the cotangent line at
the $(n+1)$-st marked point, and the inverted Euler class is expanded
as a power series in $\zeta^{-1}$ as usual in equivariant
localization. For $n = 1$ there is an additional term, equal to
$\alpha_1$, arising from the situation that there is no bubble
component attached at $0$.
\end{definition} 

Let $\pi: I_X \to X$ denote the canonical projection and $\pi^* \gamma
\in H^2(I_X)$ the pull-back of the class $\gamma \in H^2(X)$.

\begin{lemma} {\rm (Properties of localized graph potentials)} For $\alpha \in H(X)$, $\gamma \in H^2_G(X)$,
\begin{enumerate} 
\item {\rm (Duality)} $ \tau_{X,+}(\alpha,q,\zeta) =
  \tau_{X,-}(\alpha,q,-\zeta) .$
\item {\rm (Pairing)} $ \label{taupair2}
  \tau_X^{\C^\times}(\alpha,\gamma,q,\zeta) = \int_{I_X}
  \tau_{X,-}(\alpha,q,\zeta) \cup \tau_{X,+} (\alpha, qe^{\zeta
    \gamma},\zeta) \cup \exp(\pi^* \gamma).$
\end{enumerate} 
\end{lemma} 

\begin{proof} 
The pairing formula follows from the virtual localization formula
\cite{gr:loc} for the $\C^\times$-action on $\ovl{\M}_{0,n}(\P,X)$
applied to $ \tau_X^{\C^\times}(\alpha,\gamma,q,\zeta) $.  In order to
apply the virtual localization formula, one needs to know that
$\ovl{\M}_{0,n}(\P,X)$ embeds in a non-singular Deligne-Mumford stack;
this follows by taking a projective embedding of $X$.  Each fixed
point component is described in Proposition \ref{fpdes}.  The integral
over the fixed point component corresponding to components of
degrees $d_-,d_+$ attached at $0,\infty$ is
%
\begin{multline} \int_{I_X} \ev_{n_-+1,-,*} \left( \mp (\zeta (\mp \zeta -
\psi_{n_-+1}))^{-1} \bigcup_{i=1}^{n_-} \ev_{i,-}^* \alpha_i \right) \cup
\exp( (d_+, \gamma) \zeta + \pi^* \gamma) \\ \cup \ev_{n_++1,+,*}
\left( \mp (\zeta (\mp \zeta - \psi_{n_+ +1}))^{-1} \bigcup_{i=1}^{n_+}
\ev_i^* \alpha_{i,+} \right) .\end{multline}
Indeed, the normal complex of each such configuration is $
T^\dual_{w_+} C \otimes T^\dual_{w_-} \hat{C}^\rho \oplus T_{w_+} C$,
corresponding to deformations of the node and attaching point to $C$
respectively; take the inverse Euler class gives the factor $ (\mp
\zeta (\mp \zeta - \psi_{n_\pm+1}))^{-1}$.  Summing over all possible
classes $d_-,d_+$ and markings $n_-,n_+$ one obtains
$$ \int_{I_X} \tau_{X,-}(\alpha,q,\zeta) \cup \tau_{X,+} (\alpha,
qe^{\zeta \gamma},\zeta) \cup \exp(\pi^* \gamma) $$
as claimed.
\end{proof}

The components of $\tau_{X,\pm}$ give solutions to the quantum
differential equation \cite[(4)]{qk1} for the Frobenius manifold
associated to the Gromov-Witten theory of $X$ because of the
topological recursion relations, see Pandharipande \cite{pa:ra}.


\subsection{Localized gauged graph potentials} 
\label{loc2}

In this section we define a gauged version of the localized graph
potential.  We show that the gauged graph potential factorizes as a
pairing between contributions arising from the fixed points of the
$\C^\times$-action on $C = \P$ at $0$ and $\infty$.  We begin by
introducing the gauged version of the Liouville class, which was
introduced in special cases in \cite{gi:eq}.

\begin{definition}  {\rm (Liouville class and invariants with Liouville insertions)} 
Any class $\gamma \in H_G^2(X)$ gives rise to an equivariant class
\begin{equation} \label{lgamma}
\lambda(\gamma) := \exp(p_* ( e^* \gamma \cup e_C^*
       [\omega_{C,\C^\times}])) \in H_{\C^\times}(\ovl{\M}^G_n(C,X))
       .\end{equation}
(Here $H_{\C^\times}(\ovl{\M}^G_n(C,X))$ denotes equivariant
cohomology with formal power series coefficients, so that the
exponential is well defined.)  Inserting the class \eqref{lgamma} in
the integrals gives rise to gauged trace maps with Liouville
insertions
$$ \tau^{G,\C^\times,n}_X: H_G(X)^{n} \times H_{\C^\times}(
\ovl{\M}_n(C)) \times H^2_G(X) \to \Q[[\zeta]] $$
$$ (\alpha, \beta,\gamma) \to \sum_{d \in H_2^G(X,\Z)} q^d 
  \int_{[\ovl{\M}_n^G(C,X,d)_{\C^\times} \to B\C^\times]} \ev^* \alpha
  \cup f^* \beta \cup \lambda(\gamma).$$
\end{definition} 

The resulting potential, as in Givental \cite{gi:eq}, admits a
``factorization'' in terms of contributions to the fixed point formula
near $0$ and $\infty$ in $\P$; the statement and proof take the
remainder of this subsection.  First we describe the fixed point locus of the action.

\begin{definition} [Clutching construction for gauged maps from $\P$]  \label{clutchingdef}
We give a clutching construction for gauged maps, generalizing that of
bundles over the projective line.  Below we will show that all
$\C^\times$-fixed points arise from this clutching construction.
Given one-parameter subgroups $\phi_\pm: \C^\times \to G$ let
$X^{\phi_\pm}$ denote the locus of points with limits,
$$X^{\phi_\pm} := \left\{ x \in X \ | \ \exists \lim_{z \to 0}
\phi_\pm(z) x \right\} .$$
If $X$ is projective then $X^{\phi_\pm} = X$ but if $X$ is linear then
$X^{\phi_\pm}$ is the sum of the positive weight spaces.  Let
$P(\phi_+,\phi_-)$ denote the bundle over $\P^1$ formed from trivial
bundles over $\C$ with clutching function $\phi_+(z)
\phi_-(z^{-1})^{-1}$,
$$ P(\phi_+,\phi_-) = (\C \times G) \cup_{\phi_+ \phi_-^{-1}} (\C
\times G) .$$
For $x \in X^{\phi_+} \cap X^{\phi_-}$ let $u(\phi_+,\phi_-,x)$ denote
the section of $P(\phi_+,\phi_-) \times_G X$ given by
$$ (r_\pm^* u(\phi_+,\phi_-,x))(z) = \phi_\pm(z) x, z \in \C^\times $$
where $r_\pm$ is restriction to the open subsets isomorphic to $\C$
near $0$ resp. $\infty$.

A more general construction is necessary to handle orbifold case,
which involves rational one-parameter subgroups.  Suppose that $k$ is
an integer, $\pi: \C^\times \to \C^\times$ is a $k$-fold cover,
$\theta$ is a $k$-th root of unity, and $\widetilde{\phi}_\pm:
\C^\times \to G$ are one-parameter subgroups such that
$\widetilde{\phi}_\pm(\theta^i)$ fixes $x$ for all $i$,
$\widetilde{\phi}_+ \widetilde{\phi}_-^{-1}$ admits a $k$-th square
root $\phi: \C^\times \to G$.  Then
$$ P(\widetilde{\phi}_+,\widetilde{\phi}_-) = (\C \cup G) \cup_{\phi} (\C \cup G),
\quad (r_\pm^* u(\phi_+,\phi_-,x))(z) = \widetilde{\phi}_\pm(z^{1/k}) x $$
define a bundle-with-section fixed up to automorphism by the
$\C^\times$-action.
\end{definition}  

\begin{lemma}  \label{clutching} 
{\rm (Every fixed point arises from clutching)} Any $\C^\times$-fixed
element $[P,u] \in \M^G(\P,X)^{\C^\times}$ such that $u(z)$ has finite
stabilizer for generic $z$ is of the form $P =
P(\widetilde{\phi}_+,\widetilde{\phi}_-), u =
u(\widetilde{\phi}_+,\widetilde{\phi}_-,x)$ for some
$\widetilde{\phi}_+,\widetilde{\phi}_-,x \in X^{\widetilde{\phi}_-}
\cap X^{\widetilde{\phi}_+}$ as in Definition \ref{clutchingdef}.
\end{lemma} 

\begin{proof}  Suppose that $x$ has generic stabilizer of order $k$ and  
let $P \to \P$ be a bundle with section $u: \P^1 \to P \times_G X$
that is $\C^\times$-fixed up to automorphism.  For any $w \in
\C^\times$ let $m(w): \P \to \P$ denote the action of $w$.  By
assumption for any $w \in \C^\times$ there exists an isomorphism
$\phi(w) \in \Hom(P, m(w)^*P)$ so that (denoting $\phi(w): P(X) \to
m(w)^*P(X)$ with the same notation) we have $\phi(w) \circ u = m(w)
\circ u .$ The automorphism $\phi(w)$ is unique up to an element of
the finite order stabilizer of $u$ in each fiber.  In local
trivializations of $P$ near $0,\infty$ the automorphism is given by a
map $\phi_\pm : \C \to G$ and the section is given by $u(z) = m(z)
u(1) = \phi_\pm(z) u(1)$.  Furthermore, $\phi_\pm(z)$ is unique up to
an element of the stabilizer of $u(1)$ which implies that $\phi_\pm$
lifts to one-parameter subgroups $\widetilde{\phi}_\pm:
\widetilde{\C}^\times \to G$ for some finite cover $\pi:
\widetilde{\C}^\times \to \C^\times$.  Hence
$ u(z) = \widetilde{\phi}_{\pm}(\ti{z}) u(1) , \ti{z} \in
\pi^{-1}(z).$
The transition map between these two trivializations preserves the
section and is therefore of the form
$ (z,g) \mapsto (z,\widetilde{\phi}_+(\ti{z}) \widetilde{\phi}_-^{-1}(\ti{z})g) .$
The statement of the Lemma follows. 
\end{proof} 

To investigate the stability of a map formed by the clutching
construction in Lemma \ref{clutching}, we restrict to the case that
$G$ is a torus with Lie algebra $\g$ and weight lattice $\Lambda^\dual
\subset \g^\dual$.

\begin{lemma}  \label{clutching2}  {\rm (Semistability 
    of gauged maps formed by clutching)} Suppose that $\widetilde{\phi}_\pm:
  \widetilde{\C}^\times \to G$ and $x \in X^{\widetilde{\phi}_+} \cap X^{\widetilde{\phi}_-}$.
  For $\rho$ sufficiently large, the pair $(P = P(\widetilde{\phi}_-,\widetilde{\phi}_+),u =
  u(\widetilde{\phi}_-,\widetilde{\phi}_+,x))$ given by the clutching construction is Mundet
  semistable iff $x$ is semistable.
\end{lemma} 

\begin{proof}   
With $(P,u)$ as in the statement of the Lemma, the slope inequality $
\mu(\sigma,\lambda) \leq 0$ holds for all $\sigma,\lambda$ for $\rho$
sufficiently large iff $u$ is semistable at a generic point in the
domain.  Since $u$ is $\C^\times$-fixed, it suffices to check the
semistability of $u$ at $z =1 $ in the local chart near $0$, hence the
condition in the Lemma.
\end{proof}  

\begin{corollary} \label{each}
{\rm (Clutching description of the circle-fixed gauged maps)} Suppose
that $G$ is a torus.  For $\rho$ sufficiently large:
\begin{enumerate} 
\item Each component of $\M_2^G(\P,X,d)^{\C^\times}$ with markings at
  $0,\infty$ is isomorphic to a subset of $X \qu G$ with evaluation
  maps given by $\lim_{z \to 0} \widetilde{\phi}_\pm(z) x$ for some one-parameter
  subgroups $\widetilde{\phi}_\pm: \C^\times \to G$.
\item The fixed point set $\ovl{\M}_n^G(\P,X,d)^{\C^\times}$ is
  isomorphic to a union of quotients
\begin{equation} \label{fiber} 
 \left(\ovl{\M}_{0,n_-+1}(X,d_-) \times_{\ovl{I}_{X \qu G}}
 \M_2^{G,\fr}(\P,X,d_0)^{\C^\times} \times_{\ovl{I}_{X \qu G}} \ovl{\M}_{0,n_++1}(X,d_+)
 \right) \qu G^2 \end{equation}
for some $d_- + d_0 + d_+ = d $ and $n_- + n_+ = n$, where the
stability condition is induced from that on the middle factor.
\item The restriction of $\lambda(\gamma)$ to a fixed point component
  of $\ovl{\M}_n^G(\P,X)$ in \eqref{fiber} is equal to $\exp(
  \ovl{\gamma} + (d_+ + \phi_+, \gamma)\zeta)$ where $\ovl{\gamma}$ is
  the image of $\gamma$ under $H_G^2(X) \to H(\ovl{I}_{X \qu G})$ and
  $\phi_+ \in \g_\Q \cong H_2^G(X,\Q)$ is considered an element of
  $H_2^G(X)$ via the push-forward $H(BG) \to H_2^G(X)$.
\end{enumerate} 
\end{corollary} 

\begin{proof} (a) By Lemma  \ref{clutching}, 
the fixed points correspond to data
$(\widetilde{\phi}_+,\widetilde{\phi}_-,x)$ such that the
corresponding sections $\widetilde{\phi}_\pm(z)x$ extend over $0$.  By
Lemma \ref{clutching2} the value of the section over the open orbit
must be semistable, which proves the claim.  The description in (b)
includes the components $C_-,C_+$ attached to $0,\infty$ and is
immediate.  For (c) the class $[\omega_{C,\C^\times}]$ restricts to
$0,\zeta$ respectively at $0,\infty$.  The integral over the component
$C_+$ attached to $\infty$ is therefore $(\gamma,d_+)\zeta$.  The
integral over the principal component $C_0$ can be computed by
$\C^\times$-localization: Since $\C^\times$ acts on the fiber at
$\infty$ via the one-parameter subgroup $\phi_+$, the restriction of $
e^* \gamma \cup e_C^* [\omega_{C,\C^\times}]$ to the node at $\infty$
in the universal curve is $ (\gamma + (\gamma,\phi_+) \zeta ) \zeta$.
After dividing by the Euler class, one obtains that the integral over
the point $\infty$ is $\gamma + (\gamma, \phi_+) \zeta$.  Part (c)
follows.
\end{proof} 

\begin{example} {\rm (Projective space quotient)}  Let $X = \C^k$ with $G = \C^\times$ acting 
diagonally.  There are no holomorphic curves in $X$, hence the classes
$d_\pm$ of the bubble components attached to $0,\infty$ always vanish.
The moduli stack of gauged maps of class $d \in H_2^G(X,\Z) \cong \Z$
is isomorphic to $\P^{kd + k - 1}$, the projective space of $k$-tuples
of polynomials in two variables of degree $d$.  The group $\C^\times$
acts by pull-back, with fixed point set $\M(\P,X,d)^{\C^\times}$ the
union of projective spaces of $k$-tuples of homogeneous polynomials of
some degree $i = 0,\ldots, d$, each isomorphic to $\P^{k-1}$.
Identifying $H_2^G(X) \cong \Z$ we have $\phi_- = i, \phi_+ = d-i$,
and the isomorphism is given by evaluation at a generic point.  The
Liouville class is the usual Liouville class on $\P^{k-1}$, times an
equivariant correction $\exp( (\gamma,\phi_+) \zeta))$; this class
already appears in Givental \cite{gi:eq}.
\end{example}  

Let us reformulate the description of the fixed point set in Corollary
\ref{each} following the ``factorization philosophy'' as follows.
Given $\ti{\phi}_\pm,d_\pm$ as above and $x$ with order of stabilizer
$k_\pm$, let
$$F^G_{n_\pm}(\widetilde{\phi}_\pm,d_\pm) := \{ ([u],x) \in
\ovl{\M}_{0,n_\pm+1}(X,d_-) \times X \ | \ u(z_{n_\pm + 1}) = \lim_{z
  \to \pm \infty} \widetilde{\phi}_\pm(z) x \} \qu G .$$
Since $x$ is stabilized by $\widetilde{\phi}_\pm(\theta)$, where
$\theta$ is a $k$-th root of unity, we have natural maps
\begin{equation} \label{fiber2}
 F^G_{n_\pm}(\widetilde{\phi}_\pm,d_\pm) \to \ovl{I}_{X \qu G}, \quad
 (u,x) \mapsto [x, \widetilde{\phi}_\pm(\theta)] .\end{equation}
Denote by
\begin{equation} \label{allF} F^G_{n_\pm}(d) := 
\cup_{\phi_\pm + d_\pm =d} F^G_{n_\pm}(\widetilde{\phi}_\pm,d_\pm) .\end{equation}
Corollary \ref{each} implies 
\begin{equation} \label{factor}  
 \ovl{\M}_n^G(\P,X,d)^{\C^\times} \cong \bigcup_{d_- + d_+ = d}
 \bigcup_{n_- + n_+ = n} F^G_{n_-}(d_-) \times_{\ovl{I}_{X \qu G}}
 F^G_{n_+}(d_+). \end{equation}
We may view the factorization of the fixed point sets as a nodal
degeneration as follows.  Consider a degeneration of $\P^1$ to a a
nodal curve with two components, each projective weighted lines
$\P(1,k)$ with node at the orbifold singularity $B\Z_k$.  Let
$P(\phi), P(\phi_+), P(\phi_-)$ denote the (possibly orbifold) bundles
defined by clutching maps $\phi,\phi_+,\phi_-$.  Then $P(\phi)$
degenerates to a principal bundle over the nodal line with
restrictions $P(\phi_+)$ and $P(\phi_-)$.  Each $\C^\times$-fixed
section $u$ degenerates to a pair of sections $(u_-,u_+)$ of
$P(\phi_-) \cup P(\phi_+)$, given by $x$ in the trivializations near
the node and $\phi_\pm(z) x$ in the trivializations near $0$ in the
two copies of $\P(1,k)$.  The degeneration description implies the
following splitting of the normal complex.  Let
$$N(\ovl{\M}_n^G(\P,X,d)^{\C^\times}) \ \text{resp.} \ N_- :=
N(F^G_{n_-}(d_-)) \ \text{resp.} \ N_+ := N(F^G_{n_+}(d_+))$$
denote the normal complex of $\ovl{\M}_n^G(\P,X,d)^{\C^\times}$ resp.
$F^G_{n_-}(d_-)$ resp.  $F^G_{n_+}(d_+)$ in $\ovl{\M}_n^G(\P,X,d)$
resp. $\ovl{\M}_{n_-}^G(\P(1,k),X,d_-)$ resp.
$\ovl{\M}_{n_+}^G(\P(1,k),X,d_+)$.  Deforming the node gives rise to an
embedding $F^G_{n_-}(d_-) \times_X F^G_{n_+}(d_+) \to
\ovl{\M}_n^G(\P,X,d)^{\C^\times}$, and the pullback of the $K$-class of
the normal complex is independent of the deformation parameter.  It
follows that there is an isomorphism in $K$-theory
\begin{equation} \label{factorN} [N(\ovl{\M}_n^G(\P,X,d)^{\C^\times})] =  [N_-] \oplus [N_+] .\end{equation}
Explicitly for any type with more than two components in the domain,
the normal complex receives contributions from deformations of the
map, deformations of the node at the principal component, and
deformations of the attaching point to the principal component,
assuming there is some non-trivial component attached: In $K$-theory
\begin{equation} \label{Npm} 
N_\pm \cong (Rp_* e^*( T(X/G))))^{\on{mov}} \oplus \left( T^\dual_{w_+} C
\otimes T^\dual_{w_-} \hat{C}^\rho \right) \oplus T_{w_+}
C \end{equation}
where $ (Rp_* e^*(T(X/G)))^{\on{mov}}$ is the moving part (under the
action of $\C^\times$) of the index of the tangent complex of $X/G$
and $w_\pm \in \hat{C}^\rho$ are the preimages of the node connecting
to the principal component at $0$ in the normalization $\hat{C}^\rho$,
so that $w_+ = 0$ in the principal component identified with $C$.  The
first factor in \eqref{Npm} represents deformations of the map, the
second deformation of the node, and the third the deformation of the
attaching point to the principal component.  The Euler class is
$$ \Eul(N_\pm) = \Eul((Rp_* e^*( T(X/G))))^{\on{mov}} ) 
(\mp \zeta) (\mp
\zeta - \psi) $$
where $\psi$ is the cotangent line of the node of the component
attached at $0 \in \P$.  

We define the localized gauged graph potentials by twisted integration
over the fixed point sets above.  Pushforward over the map
\eqref{fiber2} induces a map in equivariant cohomology
\begin{equation} \label{compose2}
 \ev_\infty^* (\ev_1^* \times \ldots \times \ev_n^*) : H_G(X)^{\otimes
   n} \to H_{\C^\times}(I_{X \qu G}) .\end{equation}

\begin{example} \label{vspaces} {\rm (Vector spaces)}  
In the case that $X$ a vector space, the map to $X$ is homotopically
trivial and \eqref{compose2} may be identified with the map
\begin{equation} \label{factor3}
\Psi^{\phi_\pm}: H_G(X)^n \to H_{\C^\times}(X \qu G)\end{equation}
given by cup product, pull-back $H_{G \times \C^\times}(X) \to H_G(\on{pt}) \to
H_{G \times \C^\times}(X)$ under the map induced by the constant
$\C^\times$-invariant map given by multiplication by zero and the Kirwan map
$H_{G \times \C^\times}(X) \to H_{\C^\times}(X \qu G)$.  The map \eqref{factor3}
may be computed explicitly using naturality of the quotient
construction as follows.  Composition of the action with the group
homomorphism
$$ \varphi_\pm: \ G \times \C^\times \to G \times \C^\times, \quad (g,z) \mapsto
(z^{\phi_\pm} g, z) $$
makes the action of $\C^\times$ on $X$ trivial and maps the subgroup $G
\times \{ 1 \}$ to $G \times \{ 1\}$.  The quotient map $H_{G \times
  \C^\times}(X) \to H_{\C^\times}(X \qu G)$ is, for this twisted action,
independent of the $\C^\times$-equivariant parameter.  By naturality,
\eqref{factor3} is equal to the composition of the maps
$$ H_G(X)^n \to H(B(G \times \C^\times)) \to H(B(G \times \C^\times)) \to H_{G
  \times \C^\times}(X) \to H_{\C^\times}(X \qu G) $$
where the second map is induced by $\varphi_\pm$ and the action on $H_{G
  \times \C^\times}(X)$ is trivial.  After the identifications
$$H(BG) \cong \Sym(\g^\dual) , \quad H(B(G \times \C^\times)) \cong
\Sym(\g^\dual \oplus \C) $$
we obtain a description of \eqref{factor3} as the map
$$ \Sym(\g^\dual)^n \to H_{\C^\times}( X \qu G), \quad (p_1,\ldots,
p_n)(\cdot) \mapsto (\kappa_X^G |_{q = 0})(p_1 \ldots p_n)(\cdot +
\phi_\pm) $$
where $(\cdot + \phi_\pm)$ denotes translation by $\phi_\pm$ and
$\kappa_X^G |_{q= 0}$ is the classical Kirwan map.
\end{example}

\begin{definition} 
 {\rm (Localized Gauged Graph Potentials)} The {\em localized gauged
   graph potentials} $\tau_{X,\pm}^G$ are the integrals over
 $F^G_{n_\pm}(d_\pm)$ of \eqref{allF}
\begin{multline}
 \tau_{X,\pm}^G: QH_G(X) \to QH(X \qu_\rho G)[\zeta,\zeta^{-1}]], \quad
 \tau_{X,\pm}^G(\alpha,q,\zeta) = \sum_{n \ge 0} (1/n!)
 \tau_{X,\pm}^{G,n}(\alpha,\ldots, \alpha,q,\zeta)
 \\ \tau_{X,\pm}^{G,n}(\alpha_1,\ldots,\alpha_n,q,\zeta) = \sum_{d}
 q^d \ev_{\infty,*} (\ev_1^* \alpha_1 \cup \ldots \cup \ev_n^*
 \alpha_n \cup \Eul(N_\pm)^{-1}).
 \end{multline}
\end{definition} 

\begin{proposition} {\rm (Properties of localized gauged potentials)}
\begin{enumerate} 
\item {\rm (Duality)} $ \tau_{X,+}^G(q,\alpha,\zeta) =
  \tau_{X,-}^G(q,\alpha, - \zeta) .$
\item {\rm (Pairing)} $\lim_{\rho \to \infty}
  \tau_X^{G,\C^\times}(\alpha,\gamma,q,\zeta)= \int_{\ovl{I}_{X \qu G}}
  \tau_{X,-}^G(\alpha,q,\zeta) \cup \tau_{X,+}^G(\alpha,q e^{\gamma
    \zeta},\zeta )) \cup \exp( \ovl{\gamma}).$
\end{enumerate}
\end{proposition} 

Here the action $e^{\zeta \gamma}$ on $\Lambda_X^G[[\zeta]]$ is
$$ \left( f(q) = \sum c_d q^d \right) \mapsto \left( f( q \exp( \zeta
\gamma)) = \sum c_d q^d \exp( \zeta (\gamma,d)) \right) .$$

\begin{proof}    (a) The fixed point sets $F^G_n(\phi_-,d_-)$ and $F^G_n(\phi_+,d_+)$ are 
isomorphic, and $\C^\times$-equivariantly so after twisting by the
automorphism $z \mapsto 1/z$.  Similarly the complexes $N_-,N_+$ are
isomorphic up to this twisting.  The first claim follows.  (b) is a
consequence of virtual localization applied to the stack
$\ovl{\M}_n^G(\P,X)$ and \eqref{factor}, \eqref{factorN}, and part (c)
of Corollary \eqref{each}.  In order to apply the virtual localization
formula one needs to know that $\ovl{\M}_n^G(\P,X)$ embeds in a
non-singular Deligne-Mumford stack.  For this consider an embedding $G
\to GL(k)$ and $G$-equivariant embedding $X \to \P^{l-1}$ for some
$l$.  As in the proof of \cite[Proposition 5.12]{qk2},
$\ovl{\M}_n^G(\P,X)$ embeds in the moduli stack
$\ovl{\M}_{g,0}(\UU^{\fr,\quot}(C,F) \times_G X,\ti{d})/\Aut(F)$ for a
suitable sheaf $F$.  The latter embeds in
$\ovl{\M}_{g,0}(\UU^{\fr,\quot}(C,F) \times_G
\P^{l-1},\ti{d})/\Aut(F)$.  Now $\UU^{\fr,\quot}(C,F) \times_G
\P^{l-1}$ is an $\Aut(F)$-equivariant quasiprojective scheme and so
embeds in $\P^N$ for sufficiently large $N$, hence
$\ovl{\M}_{g,0}(\UU^{\fr,\quot}(C,F) \times_G \P^{l-1},\ti{d})/\Aut(F)$
embeds $\Aut(F)$-equivariantly in $\M_{g,0}(\P^N)$.  Since
$\ovl{\M}_{g,0}(\P^N)/\Aut(F)$ is a non-singular Deligne-Mumford stack,
the claim follows.
\end{proof} 

\begin{example} {\rm (Localized gauged graph potential for toric quotients)} 
\label{Ifunction} 
Let $G$ be a torus acting on a vector space $X$ is a vector space with
weights $\mu_1,\ldots,\mu_k$ and weight spaces $X_1,\ldots,X_k$ with
free quotient $X \qu G$.  Let $D_j \in H^2(X \qu G)$ denote the
divisor class corresponding to $\mu_j$.  For any given class $\phi \in
H_2^G(X,\Z) \cong \g_\Z$, the loci $X^\phi$ are sums of weight spaces
$$ X^\phi := \left\{ x \ | \ \exists \lim_{z \to 0} \phi(z) x \right\}
= \oplus_{\mu_j(\phi) \ge 0 } X_j .$$
Since there are no non-constant stable maps to $X$, $F^G(\phi,0)$ is
isomorphic to $ X^\phi \qu G$ under evaluation at any generic point.
The domain of any gauged map without markings is irreducible and the
normal complex to $F^G(\phi,0)$ is the moving part $Rp_* e^*(
T(X/G))^{\on{mov}}$.  This splits as a sum of $\mu_j(\phi)$ copies of
$X_j$ with weights $1,\ldots,\mu_j(\phi)$ for $\mu_j(\phi)$ negative,
and $-\mu_j(\phi) - 2$ copies of the normal complex for $\mu_j(\phi)
\leq 0$ with weights $\mu_j(\phi) + 1,\ldots, -1$.  Putting this
together with the normal bundle of $X^\phi \qu G$ in $X \qu G$ and
replacing $\phi$ with $d$ we obtain
\begin{equation} \label{G0} \tau_{X,-}^{G,0}(\zeta,q) = \sum_{d \in H_2^G(X)} q^d 
\frac{ \prod_{j=1}^k \prod_{m = -\infty}^{0} (D_j + m\zeta) }{
  \prod_{j=1}^k \prod_{m=-\infty}^{\mu_j(d)} (D_j + m \zeta) }.
\end{equation}
Note that the terms with $X^d \qu G = \emptyset$ contribute zero in
the above sum, since in this case the factor in the numerator $
\prod_{\mu_j(d) < 0} D_j$ vanishes.  The function
$\tau_{X,-}^G(\alpha,\zeta,q)$ for $\alpha \in H_G(X)$ can be computed
as follows.  Since there are no non-constant holomorphic spheres in
$X$, the evaluation maps $\ev_1,\ldots, \ev_{n}$ are equal on
$\ovl{\M}^G_n(\C_\pm,X)^{\C^\times}$.  It follows that the pushforward
of $\ev^* \alpha^{\otimes n}/ \mp \zeta (\mp \zeta - \psi)$ is equal
to
\begin{equation} \label{integral} 
  \alpha^{n} (\zeta)^{-2} \int_{[\ovl{M}_{0,n+1}]} (\psi_{n+1}/\mp
  \zeta)^{n-2} .\end{equation}
This integral can be computed iteratively by pushing forward under the
maps $f_i$ forgetting the $i$-th marked point for $i \leq n$.  We have
the relation for the first Chern class $\psi_{n+1}$ of the cotangent
line at the last marked point in $\ovl{\M}_{0,n+1}$
$$ \psi_{n+1} = f_i^* \psi_{n} + [D_{ \{ 0, i, n+1 \} \cup \{
    1,\ldots, \hat{i}, \ldots, n \}}] \in H(\ovl{M}_{0,n+1}) .$$
  The divisor class is degree one in any fiber of the forgetful map
  $f_i$ and it follows that the integral \eqref{integral} is equal to
  $(\alpha/\zeta)^n$.  By Example \ref{vspaces} this implies that for
  $\alpha \in H_G(X)$
\begin{equation} \label{G0} \tau_{X,-}^G(\alpha,\zeta,q) = 
\sum_{d \in H_2^G(X)} q^d \exp( \Psi^d(\alpha)/\zeta) \frac{
  \prod_{j=1}^k \prod_{m = -\infty}^{0} (D_j + m\zeta) }{
  \prod_{j=1}^k \prod_{m=-\infty}^{\mu_j(d)} (D_j + m \zeta) }.
\end{equation}
Thus $\tau_{X,-}^G$ is the generalization of Givental's {\em
  $I$-function}, see \cite{giv:tmp}, considered in
Ciocan-Fontanine-Kim \cite[Section 5.3]{ciocan:bigI}.
\end{example} 

\subsection{Localized adiabatic limit theorem}

We prove the refinement \cite[Theorem 1.6]{qk1} of the adiabatic limit
Theorem \cite[Theorem 1.5]{qk1} by comparing the fixed point
contributions to the graph potentials.  Recall the stack
$\ovl{\M}_{n,1}^G(\P,X)$ of curves with scaling from
\eqref{scaledinv}.  We claim the invariants of \eqref{scaledinv} also
have $\C^\times$-equivariant extensions:

\begin{lemma} {\rm (Existence of a circle action on the master space)}   
The $\C^\times$-action on $\ovl{\M}_n^G(\P,X)$ extends to
$\ovl{\M}_{n,1}^G(\P,X)$.  
The action on the substack
$\ovl{\M}_n^G(\P,X)_\infty$ is the action induced from the action of
$\C^\times$ on the factors $\ovl{\M}_{i_j,1}^G(\bA,X)$ and
$\ovl{\M}_r^G(\P, X \qu G)$.
\end{lemma} 

\begin{proof}  The
$\C^\times$-action on $\P$ induces an action on stable maps to $\P
  \times X$, by composition, and on the relative dualizing sheaf.
  Hence $\C^\times$ acts on $\ovl{\MM}_{n,1}^G(\P,X)$.  The stability
  condition is unchanged, since the Hilbert-Mumford weights for two
  objects related by the $\C^\times$-action are in one-to-one
  correspondence.  This implies that $\C^\times$ acts on the
  semistable locus $\ovl{\M}_{n,1}^G(\P,X)$.  For the substack
  $\ovl{\M}_n^G(\P,X)_\infty$ consisting of fiber products of
  $\ovl{\M}_{i_j,1}^G(\bA,X)$ and $\ovl{\M}_r^G(\P, X \qu G)$, the
  action is given by pull-back of sections under the action by
  rotation on $\P$.  The relative dualizing sheaf is identified with
  the dualizing sheaf by trivialization of the cotangent line
  $T^\dual_0 \P$.  Since $\C^\times$ acts with weight one on
  $T^\dual_0 \P$, after trivialization of the cotangent line it acts
  by scalar multiplication on the sections of the projectivized
  dualizing sheaf in the objects of $\ovl{\M}_{i_j,1}^G(\bA,X)$, as
  claimed.
\end{proof}  

The fixed point set for the action factorizes as follows.  Let
$F_{n,1}^G(d) \subset \ovl{\M}_{n,1}^G(\P,X,d)$ denote the fixed point
set of class $d \in H_2^G(X)$.  We wish to express $F_{n,1}^G(d)$ as a
fiber product of ``contributions from $0,\infty$''.  Let
$\ovl{\M}^G_{n_\pm,1}(\P(1,k), X,d)^{\C^\times}$ denote the
$\C^\times$-fixed locus in $\ovl{\M}^G_{n_\pm,1}(\P(1,k),X)$.  Let
$$ F^G_{n_\pm,1}(d) \subset \ovl{\M}^G_{n_\pm,1}(\P(1,k),
X,d)^{\C^\times}$$
denote the locus of bundles with sections that are constant in some
trivialization in a neighborhood of $B\Z_k$ and $n_\pm$ markings map
to $0 \in \P(1,k)$ under the projection $\P(1,k) \times X \to
\P(1,k)$.  Both conditions are preserved under limits.  It follows
that $F^G_{n_\pm,1}(d)$ is a proper Deligne-Mumford stack.  It admits
a relative perfect obstruction theory over
$\ovl{\M}_{n_\pm,1}(\P(1,k))$ induced from the relative perfect
obstruction theory on $\ovl{\M}^G_{n_\pm,1}(\P(1,k), X,d)^{\C^\times}$.
Consider the evaluation map $F^G_{n_\pm,1}(d) \to \ovl{I}_{X \qu G}$ at
$B\Z_k \subset \P(1,k)$.  For any $g \in G$ of finite order $k$ denote
by $F_{n_\pm,1}^G(d,[g]) \subset F_{n_\pm,1}^G(d)$ the locus of
bundles-with-sections whose sections take values in the twisted sector
corresponding to $[g] \in G/\Ad(G)$ at $B\Z_k$.  Consider the map
$$ \pi_\pm: F^G_{n_\pm,1}(d_\pm,[g]) \to \M_{0,1}(\P(1,k)) \cong \P
.$$
The fiber $\pi^{-1}_\pm(\rho)$ of $F_{n_\pm,1}^G(d,[g])$ over a
non-zero scaling $\rho$ is by Section \ref{loc1}
\begin{equation} \label{nonzero}
 \pi^{-1}(\rho) = \bigcup_{d_- + \phi_- = d,
   \widetilde{\phi}_-(\theta) = g}
 F_{n_\pm}^G(d_-,\phi_-) \end{equation}
where $\theta$ is a $k$-th root of unity.  The locus
$\pi^{-1}_\pm(\infty)$ in $F_{n_\pm,1}^G(d,[g])$ with infinite scaling
is by Section \ref{loc2} a union of fiber products
\begin{equation}  \label{inf} \left( \prod_{j=1}^r \ovl{\M}_{i_j,1}^G(\bA,X) \right) \times_{(\ovl{I}_{X
    \qu G})^r} \times \ovl{\M}_{0,r + 1}(X \qu G,d) \times_{\ovl{I}_{X \qu G}}
  (X^g \qu Z_g) .\end{equation}
The identification is given by the map which attaches affine gauged
maps to $0 \in \P(1,k)$.  From the descriptions in Sections
\ref{loc1}, \ref{loc2} we obtain
$$ F_{n,1}^G(d) = \bigcup_{d_- + d_+ = d, n_- + n_+= n} \bigcup_{[g]
  \in G/\Ad(G)} F_{n_-,1}^G(d_+,[g]) \times_{\ovl{I}_{X \qu G} \times \P}
F_{n_+,1}^G(d_-,[g]) $$
where the map to $\P$ is given by the scaling.  We can now complete
the proof of the localized version of the adiabatic limit theorem,
\cite[Theorem 1.6]{qk1}.

\begin{proof} (of \cite[Theorem 1.6]{qk1})
The divisor class relation $[\pi^{-1}(0)] = [\pi^{-1}(\infty)] $
implies that the integrals over \eqref{nonzero}, \eqref{inf} are
equal.  Hence for any $\alpha_\infty \in H(X^g \qu Z_g) \subset H(I_{X
  \qu G})$, we have
\begin{multline}
 \sum_{[I_1,\ldots,I_r]} \int_{[I_{X \qu G}]} (\tau_{X \qu
   G,-}^{r} (\kappa_X^{G,|I_1|} (\alpha_{I_1},1), \ldots,
 \kappa_X^{G,|I_r|} (\alpha_{I_r},1),1)\cup \alpha_\infty) \\=
 \int_{[I_{X \qu G}]} \tau_{X,-}^{G,n}(\alpha_1,\ldots,\alpha_n,1)
 \cup \alpha_\infty. \end{multline}
Summing over $n$ with $\alpha_1 = \ldots \alpha_n = \alpha$ gives
$$ \int_{[I_{X \qu G}]} (\tau_{X \qu G,-} \circ \kappa_X^G)(\alpha)
\cup \alpha_\infty = \int_{[I_{X \qu G}]} \tau_{X,-}^G(\alpha) \cup
\alpha_\infty .$$
Since this holds for any $g$ and any $\alpha_\infty$, it follows
\cite[Theorem 1.6]{qk1}
$$(\tau_{X \qu G,-} \circ \kappa_X^G)(\alpha) = \tau_{X,-}^G(\alpha)
\in QH_{S^1}(X \qu G) .$$
\end{proof} 

The localized adiabatic limit \cite[Theorem 1.6]{qk1} allows us to
deduce relations in the quantum cohomology algebras, although these
relations are rather non-explicit unless $\kappa_X^G$ is known.
Namely, recall that any differential operator annihilating the
localized graph potential $\tau_{X \qu G,-}$ defines relations, see
for example \cite{gi:eq}, \cite{iri:qdm}.  Let $V \subset QH_G(X)$
be a submanifold such that the restriction of $\kappa_X^G$ to $V$ is
an embedding.  In particular, any differential operator on $V$ pushes
forward to a differential operator on $(\kappa_X^G)_*(V)$.  We
especially have in mind the case that $V \cong QH_G^2(X)$ and is
isomorphic to $QH^2(X \qu G)$.

\begin{corollary} {\rm (Relations on quantum cohomology algebras)} 
\label{relations}  Suppose that $\Box$ is a differential operator 
on $V$.
\begin{enumerate} 
\item {\rm (Annihilation at a point)} If $\Box$ annihilates the
  restriction of $ \tau_{X,-}^G$ at $\alpha \in \widetilde{QH_{S^1}(X \qu
    G)}$, then $(\kappa_X^G)_* \Box$ annihilates $\tau_{X \qu G,-}$ at
  $\kappa_X^G(\alpha)$ and the symbol $\Sigma( (\kappa_X^G)_* \Box)$
  of $(\kappa_X^G)_* \Box$ satisfies $\Sigma( (\kappa_X^G)_* \Box)
  \star_{\kappa_X^G(\alpha)} \tau_{X \qu G,-}(\alpha) = 0$.
\item {\rm (Annihilation on small quantum cohomology)} If $V =
  QH^2_G(X)$ and the quantum cohomology $T_{\kappa_X^G(\alpha)} QH(X
  \qu G)$ is generated by $T_{\kappa_X^G(\alpha)} QH^2(X \qu G)$, then
  $\Sigma( (\kappa_X^G)_* \Box)(\alpha)$ is a relation in
  $T_{\kappa_X^G(\alpha)} QH_{S^1}(X \qu G)$.
\end{enumerate} 
\end{corollary}

\begin{proof}  Suppose $\Box$ is a differential operator as in the Corollary
part (a).  That $(\kappa_X^G)_* \Box$ annihilates $\tau_{X \qu G,-}$
follows from \cite[Theorem 1.6]{qk1}.  The relation on the principal
symbol follows from the fact that $\tau_{X \qu G,-}$ is a fundamental
solution to the quantum differential equation \cite[(4)]{qk1}, and so
differential operators transform into quantum multiplications at
$\kappa_X^G(\alpha)$. That the principal symbol defines relations
follows as in \cite{gi:eq}, \cite[Theorem 2.4]{iri:qdm}, using that
the components of $\tau_{X \qu G,-}$ generate the quantum cohomology
resp. derivatives of $\tau_{X \qu G,-}$ in the directions $ (D
\kappa_X^G)(\alpha), \alpha \in QH^2_G(X)$ under the
$QH^2_G(X)$-generation hypothesis.
\end{proof} 


In the remainder of the section we discuss the toric case, that is,
$X$ is a complex vector space and $G$ is a torus acting on $X$ so that
$X \qu G$ is a smooth proper Deligne-Mumford stack.  In this case, the
identity in \cite[Theorem 1.6]{qk1} seems to be essentially the same
as the ``mirror theorems'' in \cite{gi:eq}, \cite{lly:mp1},
\cite{iri:gmt}, \cite{coates:mirrorstacks},
\cite{ciocan:genuszerowall}.  Regarding relations in the quantum
cohomology of toric varieties, the following is introduced in Batyrev
\cite{bat:qcr}.

\begin{definition}  
The {\em quantum Stanley-Reisner ideal} is the ideal $ QSR_X^G \subset
T_0 QH_G(X), \alpha \in QH^2_G(X) $
generated by the elements
$$ \prod_{\mu_j(d) > 0} r(v_j)^{\mu_j(d)} - q^d \prod_{\mu_j(d)
  < 0} r(v_j)^{\mu_j(d)} \in T_0 QH_G(X) $$
for $d \in H^G_2(X,\Z) \cong \g_\Z $.  Similarly the {\em equivariant
  quantum Stanley-Reisner ideal} is the ideal $ {QSR}_X^{\ti{G},G}
\subset T_0 QH_{\ti{G}}(X)$ generated by the elements
$$ \prod_{\mu_j(d) > 0} v_j^{\mu_j(d)} - q^d \prod_{\mu_j(d) < 0}
v_j^{\mu_j(d)} \in T_0 QH_{\ti{G}}(X) ,$$
that is, without restriction to $\g \subset \ti{\g}$.
\end{definition}  

\begin{example} 
\begin{enumerate}
\item {\rm (Projective spaces)} For the usual action of $G = \C^\times$ on
  $X = \C^k$, we have $H^2_G(X,\Z) = \Z$ with positive generator $d =
  1$ and $r(v_j) = u$ the coordinate on $\g$ for $j =1,\ldots, k$.
  The quantum Stanley-Reisner ideal $QSR_X^G$ has generator
$$  \prod_{j=1}^k r(v_j) - q =  u^k - q $$
while the equivariant Stanley-Reisner ideal $QSR_X^{\ti{G},G}$ has
generator $ \prod_{j=1}^k v_j - q .$
\item {\rm (Weighted projective line)} Suppose that $G =\C^\times$ acts on
  $X = \C^2$ with weights $\mu_1 = 2, \mu_2 = 3$.  Then $H_2^G(X) \cong \Z$ with
  generator $d = 1$.  We have $r(v_1) = 2u$ while $r(v_2) =3u$.  The
  quantum Stanley-Reisner ideal $QSR_X^G$ has generator
$$ \prod_{j=1}^2 r(v_j)^{\mu_j(1)} - q = (2u)^2 (3u)^3 - q $$
while the equivariant Stanley-Reisner ideal
$QSR_X^{\ti{G},G}$ has generator
$$  \prod_{j=1}^2 v_j^{\mu_j(1)} - q = v_1^2 v_2^3 - q .$$
\end{enumerate} 
\end{example} 

Recall the definition of $\kappa_X^{\ti{G},G}$ from Remark
\ref{eqqkirwan}.  

\begin{theorem} \label{QSR} 
The kernel of $D_0 \kappa_X^{G,1}: T_0 QH_G(X) \to
T_{\kappa_X^G(0)} QH_{S^1}(X \qu G)$ resp.  $D_0
\kappa_X^{\ti{G},G}: T_0 QH_G(X) \to
T_{\kappa_X^{\ti{G},G}(0)} QH_T(X \qu G)$ contains $QSR_X^G$
resp. $QSR_X^{\ti{G},G}$.
\end{theorem} 

\begin{proof}   For $d \in H_2^G(X,\Z) \cong H(BG,\Z) \cong \g_\Z^\dual$ let $\Box_d$ denote the differential operator on $QH_2^G(X,\R)$ corresponding to $d$, 
$$ \Box_d = \prod_{\mu_j(d) \ge 0} \partial_j^{\mu_j(d)} - q^d
  \prod_{\mu_j(d) \le 0} \partial_j^{-\mu_j(d)} .$$
We may identify the coordinates on $QH_2^G(X,\R)$ with the quantum
parameters, using the divisor equation.  
Then the operator $\Box_d$
annihilates the function of Example \ref{Ifunction}, see for example
Iritani \cite{iri:gmt}, Cox-Katz \cite[(11.92)]{ck:ms}.  It follows
from Corollary \ref{relations} that the corresponding product of the
tangent vectors maps to zero in $T_{\kappa_X^G(0)} QH_{S^1}(X \qu G)$,
and so lies in the kernel of $D_0 \kappa_X^G$.
 \end{proof} 

\begin{remark}  {\rm (Isomorphism with the quantum Stanley-Reisner ring)} 
In joint work with Gonz\'alez \cite{gw:surject}, we show that $\kappa_X^{G,1}$ is
surjective and $QSR_X^G$ is exactly its kernel, after passing to a
suitable formal version of $QH_G(X)$, so that $T_{\kappa_X^G(0)} QH(X
\qu G)$ is canonically isomorphic to the quantum Stanley-Reisner ring.
Related computations can be found in McDuff-Tolman \cite{mct:top} and
Iritani \cite{iritani:conv}.  
\end{remark}

\def\cprime{$'$} \def\cprime{$'$} \def\cprime{$'$} \def\cprime{$'$}
  \def\cprime{$'$} \def\cprime{$'$}
  \def\polhk#1{\setbox0=\hbox{#1}{\ooalign{\hidewidth
  \lower1.5ex\hbox{`}\hidewidth\crcr\unhbox0}}} \def\cprime{$'$}
  \def\cprime{$'$} \def\cprime{$'$} \def\cprime{$'$}

\end{document}